\documentclass[11pt]{amsart}

\usepackage{amsmath,amsthm,amssymb}
\usepackage[margin=1in]{geometry}
\usepackage[bookmarks]{hyperref}
\usepackage[all,cmtip]{xy}
\usepackage{verbatim}

\newcommand{\abs}[1]{\left| #1 \right|}

\newtheorem{theorem}{Theorem}[subsection]
\newtheorem{proposition}[theorem]{Proposition}
\newtheorem{conjecture}[theorem]{Conjecture}

\newtheorem{lemma}[theorem]{Lemma}

\theoremstyle{definition}

\newtheorem{definition}[theorem]{Definition}

\newtheorem{quest}[theorem]{Question}

\numberwithin{equation}{subsection}

\newcommand{\Lie}{\operatorname{Lie}}

\newcommand{\ind}{\operatorname{ind}}
\newcommand{\Ext}{\operatorname{Ext}}

\newcommand{\Hom}{\operatorname{Hom}}

\newcommand{\ga}{\gamma}

\newcommand{\la}{\lambda}

\newcommand{\al}{\alpha}
\newcommand{\ta}{\tilde{\alpha}}

\newcommand{\si}{\sigma}

\newcommand{\St}{\operatorname{St}}
\newcommand{\Dist}{\operatorname{Dist}}

\begin{document}

\title[On Tensoring with the Steinberg Representation]
{\bf On Tensoring with the Steinberg Representation}

\begin{abstract} Let $G$ be a simple, simply connected algebraic group over an algebraically closed field of prime characteristic $p>0$. 
Recent work of Kildetoft and Nakano and of Sobaje has shown close connections between two long-standing conjectures of Donkin: one on tilting modules 
and the lifting of projective modules for Frobenius kernels of $G$ and another on the existence of certain filtrations of $G$-modules. A key question related to these conjectures 
is whether the tensor product of the $r$th Steinberg module with a simple module with $p^{r}$th restricted highest weight admits a good filtration. In this paper we verify this statement when 
(i) $p\geq 2h-4$ ($h$ is the Coxeter number), (ii) for all rank two groups, (iii) for $p\geq 3$ when the simple module corresponds to a fundamental weight and 
(iv) for a number of cases when the rank is less than or equal to five. 
\end{abstract}

\author{\sc Christopher P. Bendel}
\address
{Department of Mathematics, Statistics and Computer Science\\
University of
Wisconsin-Stout \\
Menomonie\\ WI~54751, USA}
\thanks{Research of the first author was supported in part by Simons Foundation Collaboration Grant 317062}
\email{bendelc@uwstout.edu}

\author{\sc Daniel K. Nakano}
\address
{Department of Mathematics\\ University of Georgia \\
Athens\\ GA~30602, USA}
\thanks{Research of the second author was supported in part by
NSF grant DMS-1701768}
\email{nakano@math.uga.edu}

\author{\sc Cornelius Pillen}
\address{Department of Mathematics and Statistics \\ University
of South
Alabama\\
Mobile\\ AL~36688, USA}
\thanks{Research of the third author was supported in part by Simons Foundation Collaboration Grant 245236}
\email{pillen@southalabama.edu}

\author{Paul Sobaje}
\address{Department of Mathematics \\
          University of Georgia \\
          Athens, GA 30602}
\email{sobaje@uga.edu}
\thanks{Research of the fourth author was partially supported by NSF RTG grant  DMS-1344994}
\date{\today}
\date\today

\maketitle
\section{Introduction}

\subsection{} Let $G$ be a simple, simply connected algebraic group scheme over the algebraically closed field $k$ of characteristic $p > 0$. 
Let $X$ be the set of integral weights and $X_{+}$ denote the dominant integral weights (relative to a fixed choice of a Borel subgroup). For any $\lambda\in X_{+}$, one can construct a 
non-zero module $\nabla(\lambda)=\text{ind}_{B}^{G}\lambda$ and the Weyl module $\Delta(\lambda)$. The character of these modules is 
given by Weyl's character formula. The finite dimensional simple modules $L(\lambda)$ are indexed by dominant integral weights $X_{+}$ 
and can be realized as the socle of $\nabla(\lambda)$ (and the head of $\Delta(\lambda)$). 

A central idea in this area has been the concept of good and Weyl filtrations. A $G$-module admits a {\em good filtration} (resp. {\em Weyl filtration}) if and only if it admits
a $G$-filtration with sections of the form $\nabla(\mu)$ (resp. $\Delta(\mu)$) where $\mu\in X_{+}$. Cohomological criteria have been 
proved by Donkin and Scott which give necessary and sufficient conditions for a module to admit a good filtration (resp. Weyl filtration). A module which 
admits both a good and Weyl filtration is called a {\em tilting module}. Ringel proved that (i) for every $\lambda\in X_{+}$, there is an indecomposable tilting 
module $T(\lambda)$ of highest weight $\la$ and (ii) every tilting module is a direct sum of these indecomposable tilting modules. 

Determining the characters of simple modules and tilting modules remains an open problem.  In 2013, Williamson \cite{Wi} produced families of counterexamples to the 
Lusztig conjecture for $G={SL}_{n}$ and showed that the Lusztig Character Formula (LCF) in this case cannot hold for any linear bound on $p$ relative to $h$ (the associated Coxeter number). 
It is now evident that the character formula for simple modules will be highly dependent on the prime $p$. Therefore, this makes the understanding of the behavior of various 
$G$-filtrations even more crucial. A new approach has been introduced by Riche and Williamson \cite{RW} that shows the characters of tilting modules and simple modules are given by $p$-Kazhdan-Lusztig 
polynomials that are constructed using $p$-Kazhdan-Lusztig bases.  

\subsection{} Let $\lambda\in X_+$ with unique decomposition $\lambda = \lambda_0 + p^r\lambda_1$ with $\lambda_0\in X_r$ ($p^{r}$th restricted weights) and $\lambda_1\in X_+$. One 
can define $\nabla^{(p,r)}(\lambda) = L(\lambda_0)\otimes \nabla(\lambda_1)^{(r)}$ where $(r)$ denotes the twisting of the module action by the $r$th Frobenius morphism. A $G$-module $M$ has 
a {\em good $(p,r)$-filtration} if and only if $M$ has a filtration with factors of the form $\nabla^{(p,r)}(\mu)$ for suitable $\mu\in X_+$. 
Let $\text{St}_r = L((p^r-1)\rho)$ (which is also isomorphic to $\nabla((p^r-1)\rho)$ and $\Delta((p^r-1)\rho)$) be the $r$th Steinberg module, where $\rho$ is the sum of the fundamental weights. 

The following conjecture, introduced by Donkin at MSRI in 1990, interrelates good filtrations with good $(p,r)$-filtrations via the Steinberg module. 

\begin{conjecture} \label{donkinconj} Let $M$ be a finite-dimensional $G$-module. Then $M$ has a good $(p,r)$-filtration if and only if $\operatorname{St}_r\otimes M$ has a good filtration.
\end{conjecture} 

At the same meeting, Donkin presented another conjecture that realizes the injective hull, $Q_{r}(\lambda)$, for $L(\lambda)$ over $G_{r}$ as a tilting module, where $G_r$ denotes the $r$th Frobenius kernel of $G$. 
 
\begin{conjecture}\label{tilting}  For all $\lambda\in X_{r}$, $T(2(p^{r}-1)\rho+w_{0}\lambda)|_{G_{r}}=Q_{r}(\lambda)$ where $w_{0}$ denotes 
the long element in the Weyl group $W$. 
\end{conjecture} 

In exciting recent developments, it has been shown how these conjectures are related. Kildetoft and Nakano \cite{KN} proved that Conjecture~\ref{tilting} implies the forward direction (i.e., the only if portion) of Conjecture~\ref{donkinconj} (which we will denote by ``Conjecture~\ref{donkinconj}($\Rightarrow$)''). Sobaje \cite{So} has proved that Conjecture~\ref{donkinconj} implies Conjecture~\ref{tilting}. It is well-known that Conjecture~\ref{donkinconj}($\Rightarrow$) is equivalent to $\text{St}_{r}\otimes L(\lambda)$ having a good filtration for all $\lambda\in X_{r}$. Combining these results, the following hierarchy of conjectures has now been established: 
\begin{center}
\text{Conjecture}~\ref{donkinconj} \ \ $\Rightarrow$ \ \  \text{Conjecture}~\ref{tilting} \ \ $\Rightarrow$ \ \ \text{Conjecture}~\ref{donkinconj}($\Rightarrow$).
\end{center}
Using different approaches, Andersen \cite{andersen01} and later Kildetoft and Nakano \cite{KN} verified Conjecture~\ref{donkinconj}($\Rightarrow$) when $p\geq 2h-2$. This paper is focused on the verification of $\text{St}_{r}\otimes L(\lambda)$ having a good filtration for all $\lambda\in X_{r}$ for many new cases. The reader should note that the 
connections to these various conjectures are both useful and striking. For example, if one discovers an example when $\text{St}_{r}\otimes L(\lambda)$ does not have 
a good filtration for some $\lambda\in X_{r}$ then Conjecture~\ref{tilting} would be false. 

It also should be mentioned that the verification of Conjecture~\ref{tilting} would prove the 40 year old Humphreys-Verma Conjecture about the existence of $G$-structures on injective indecomposable $G_{r}$-modules. Conjecture~\ref{tilting} holds for $p\geq 2h-2$ and the proof under this bound entails locating one particular $G$-summand of $\text{St}_{r}\otimes L(\lambda)$. It has become evident that, in order to prove either conjecture for all $p$, one needs to analyze all $G$-summands of $\text{St}_{r}\otimes L(\lambda)$. 

\subsection{} The paper is organized as follows. In Section 2, we summarize the basic definitions and fundamental results on good (resp. good $(p,r)$-) filtrations. 
The following section, Section 3, is devoted to developing sufficient conditions to guarantee that $\text{St}_{r}\otimes M$ has a good filtration for a rational $G$-module $M$. These sufficient conditions 
involve the mysterious Frobenius contraction functor studied by Gros and Kaneda \cite{GK17} and Andersen \cite{andersen17}. These results are used in Section 4 
to prove that $\text{St}_{r}\otimes L(\lambda)$ where $\lambda\in X_{r}$ has a good filtration for (i) $p\geq 2h-4$ and (ii) for all rank two groups. The reader should note that Donkin's Tilting Module Conjecture (i.e., Conjecture ~\ref{tilting}) is not known for all rank $2$ groups. Later in this section, $\text{St}_{r}\otimes L(\lambda)$ is shown to have a good filtration when $\lambda$ is a fundamental weight as long as one is not in the cases of $E_{7}$ and $E_{8}$ when $p=2$. 

Section 5 is devoted to verifying Conjecture~\ref{donkinconj}($\Rightarrow$) for many cases when the rank of $G$ is less than or equal to 5. In Section 6, we carefully 
analyze the type $A_{5}$, $p=2$ situation and verify the conjecture using new and detailed information. This is an important case because it is indicative of 
the cases of fundamental weights for $E_{7}$ and $E_{8}$ when $p=2$, where the conjecture is not yet verified. At the end of the paper in Section 7, we consider the question of whether $\St_r\otimes k[G_{r}]$ has a good filtration, where $k[G_r]$ is regarded as a $G$-module by the conjugation action. 

\subsection{Acknowledgements} The authors would like to thank Henning Andersen, Stephen Donkin and James Humphreys for their comments on an earlier version of this manuscript.


\section{Preliminaries} 

\subsection{Notation}
\noindent 
Throughout this paper, the following basic notation will be used. See \cite{rags} for a general overview of terminology.  
\begin{enumerate}

\item[(1)] $k$: an algebraically closed field of characteristic $p> 0$.

\item[(2)] $G$: a simple, simply connected algebraic group scheme over $k$, defined over $\mathbb{F}_p$ (the assumption of $G$ being simple is for convenience and the results easily generalize to $G$ reductive).

\item[(3)] $\operatorname{Dist}(G)$: the distribution algebra of $G$.

\item[(4)] ${\mathfrak g} = \operatorname{Lie}(G)$: the Lie algebra of $G$.

\item[(5)] $T$: a maximal split torus in $G$. 

\item[(6)] $\Phi$: the corresponding (irreducible) root system associated to $(G,T)$. When referring to short and long roots, when a root system has roots of only one length, all roots shall be considered as both short and long.

\item[(7)] $\Phi^{\pm}$: the positive (respectively, negative) roots.  

\item[(8)] $S = \{\alpha_1,\alpha_2,\dots,\alpha_n\}$: an ordering of the simple roots.

\item[(9)]  $B$: the Borel subgroup containing $T$ corresponding to the negative roots. 

\item[(10)] $U$: the unipotent radical of $B$.

\item[(11)] $\mathbb E$: the Euclidean space spanned by $\Phi$ with inner product $\langle\,,\,\rangle$ normalized 
so that $\langle\alpha,\alpha\rangle=2$ for $\alpha \in \Phi$ any short root.

\item[(12)] $\alpha^\vee=2\alpha/\langle\alpha,\alpha\rangle$: the coroot of $\alpha\in \Phi$.

\item[(13)] $X=X(T)=\mathbb Z \varpi_1\oplus\cdots\oplus{\mathbb Z}\varpi_n$: the weight lattice, where the 
fundamental dominant weights $\varpi_i\in{\mathbb E}$ are defined by $\langle\varpi_i,\alpha_j^\vee\rangle=\delta_{ij}$, $1\leq i,j\leq n$.

\item[(14)] $X_{+}=X(T)_{+}={\mathbb N}\varpi_1+\cdots+{\mathbb N}\varpi_n$: the dominant weights.

\item[(15)] $X_{r}=X_{r}(T)=\{\lambda\in X(T)_+: 0\leq \langle\lambda,\alpha^\vee\rangle<p^{r},\,\,\forall \alpha\in S \}$: the set of $p^{r}$-restricted dominant weights.

\item[(16)] $F:G\rightarrow G$: the Frobenius morphism. 

\item[(17)] $G_r=\text{ker }F^{r}$: the $r$th Frobenius kernel of $G$. Similarly, $B_r$, $T_r$, and $U_r$ denote the kernels of the restriction of $F^r$ to $B$, $T$, and $U$ respectively.

\item[(18)] Set $G^{(r)}=G/G_{r}$ and $B^{(r)} = B/B_r$.

\item[(19)] $W$: the Weyl group of $\Phi$. 

\item[(20)] $w_{0}$: the longest element of the Weyl group. 

\item[(21)] $\rho$: the Weyl weight defined by $\rho=\frac{1}{2}\sum_{\alpha\in\Phi^+}\alpha$.

\item[(22)] $\alpha_0$: the maximal short root.

\item[(23)] $h$: the Coxeter number of $\Phi$, given by $h=\langle\rho,\alpha_0^{\vee} \rangle+1$.

\item[(24)] $\leq$ on $X(T)$: a partial ordering of weights, for $\lambda, \mu \in X(T)$, $\mu\leq \lambda$ if and only if $\lambda-\mu$ is a linear combination 
of simple roots with non-negative integral coefficients. 

\item[(25)] $M^{(r)}$:  the module obtained by composing the underlying representation for 
a rational $G$-module $M$ with $F^{r}$.

\item[(26)] $M^*$: the $k$-linear dual module for a rational $G$-module $M$.

\item[(27)] $\la^* := -w_0\la$, $\la \in X$: the dual weight.

\item[(28)] $^{\tau}M$: the dual module $M^*$ for a rational $G$-module $M$ with action composed with the anti-automorphism $\tau: G \to G$ that interchanges positive and negative root subgroups.  

\item[(29)] $\nabla(\lambda) := \operatorname{ind}_B^G\lambda$, $\lambda\in X_{+}$: the induced module whose character is provided by Weyl's character formula.  

\item[(30)] $\Delta(\lambda)$, $\lambda\in X_{+}$: the Weyl module of highest weight $\lambda$. Here $\Delta(\lambda)\cong \nabla(-w_{0}(\lambda))^*$.

\item[(31)] $L(\lambda)$, $\la \in X_{+}$: the simple finite dimensional $G$-module with highest weight $\lambda$. 

\item[(32)] $T(\lambda)$, $\la \in X_{+}$: the indecomposable finite dimensional tilting $G$-module with highest weight $\lambda$. 

\item[(33)] $\nabla^{(p,r)}(\lambda) := L(\lambda_0)\otimes \nabla(\lambda_1)^{(r)}$, $\la \in X_{+}$, where $\lambda = \lambda_0 + p^r\lambda_1$ with $\lambda_0\in X_r$ and $\lambda_1\in X_+$.

\item[(34)] $\Delta^{(p,r)}(\lambda) := T(\hat{\lambda}_0)\otimes \Delta(\lambda_1)^{(r)}$, $\la \in X_{+}$, where $\lambda = \lambda_0 + p^r\lambda_1$ with $\lambda_0\in X_r$ and $\lambda_1\in X_+$. Here $\hat{\lambda}_0=2(p^{r}-1)\rho+w_{0}\lambda_{0}$. 

\item[(35)] $\St_r := L((p^r - 1)\rho)$: the $r$th Steinberg module.

\item[(36)] $Q_r(\la)$, $\la \in X_r$: the injective hull (or equivalently, projective cover) of $L(\la)|_{G_r}$ as a $G_r$-module.

\end{enumerate} 

\subsection{Important $G$-Filtrations} Let $M$ be a rational $G$-module. In this paper a {\em $G$-filtration} for $M$ is an increasing sequence of $G$-submodules of $M$: $M_{0}\subseteq M_{1} \subseteq \dots \subseteq M$ such that $\cup_{i}M_{i}=M$. 
We now present the definition of a good filtration and a good $(p,r)$-filtration. 

\begin{definition} Let $M$ be a $G$-module 
\begin{itemize}
\item[(a)] $M$ has a {\em good filtration} if and only if it has a $G$-filtration such that for each $i$, $M_{i+1}/M_{i}\cong \nabla(\lambda_{i})$ where $\lambda_{i}\in X_+$. 
\item[(b)] $M$ has a {\em good $(p,r)$-filtration} if and only if it has a $G$-filtration such that for each $i$, $M_{i+1}/M_{i}\cong \nabla^{(p,r)}(\lambda_{i})$ where $\lambda_{i}\in X_+$.
\item[(c)] If $M$ has a {\em good $(p,1)$-filtration}, then we say that $M$ has a good $p$-filtration.
\end{itemize}
\end{definition} 

\subsection{Good Filtrations: Cohomological Criterion} The following well-known result due to Donkin \cite[Cor. 1.3]{donkin81} and Scott \cite{Sc} 
(cf. {\cite[Proposition II.4.16]{rags}}) gives a very useful criterion to prove the existence of good filtrations.

\begin{theorem}\label{T:cohcrit} Let $M$ be a $G$-module. The following are equivalent:
\begin{itemize} 
\item[(a)] $M$ has a good filtration. 
\item[(b)] $\Ext_G^1(\Delta(\mu),M) = 0$ for all $\mu\in X_+$.
\item[(c)] $\Ext_G^n(\Delta(\mu),M) = 0$ for all $\mu\in X_+$, $n\geq 1$.
\end{itemize}
\end{theorem}

\subsection{Good Filtrations: Tensoring with the Steinberg} Kildetoft and Nakano \cite[Theorem 9.2.3]{KN} gave necessary and sufficient conditions for $\St_r\otimes M$ to admit a  good filtration (cf. \cite[Proposition 2.7]{andersen01}).

\begin{theorem} \label{T:cohom-criteria} Let $M$ be a $G$-module with $\dim \operatorname{Hom}_{G}(\Delta^{(p,r)}(\lambda),M)< \infty$ for 
all $\lambda \in X_{+}$. The following are equivalent:
\begin{itemize} 
\item[(a)] $\operatorname{St}_{r}\otimes M$ has a good filtration. 
\item[(b)] $\Hom_{G_{r}}(T(\hat{\mu}),M)^{(-r)}$ has a good filtration for all $\mu\in X_{r}$. 
\item[(c)] $\Ext_G^n(\Delta^{(p,r)}(\lambda),M) = 0$ for all $\lambda\in X_+$, $n\geq 1$.
\item[(d)] $\Ext_G^1(\Delta^{(p,r)}(\lambda),M) = 0$ for all $\lambda\in X_+$.
\end{itemize} 
\end{theorem}

Observe that if Conjecture~\ref{donkinconj} holds then Theorem \ref{T:cohom-criteria} would give cohomological criteria for a $G$-module $M$ to 
admit a good $(p,r)$-filtration.


\section{Good Filtrations on $\St_{r}\otimes M$} 

\subsection{} We will first introduce an important class of functors via the $r$th-Steinberg module that 
sends $G$-modules to $G/G_{r}$-modules. For $\mu\in X_{r}$ and a rational $G$-module $M$, set 
\begin{equation} 
{\mathcal F}_{\mu}(M)=\text{Hom}_{G_{r}}(k,\text{St}_{r}\otimes \nabla(\mu)\otimes M) \cong \Hom_{G_r}(k,\ind_{B}^{G}(\St_{r}\otimes\mu\otimes M)),
\end{equation} 
where the latter isomorphism follows from the tensor identity.
The functor ${\mathcal F}_{\mu}$ is an exact functor from $\text{Mod}(G)\rightarrow \text{Mod}(G/G_{r})$. 
Exactness of ${\mathcal F}_{\mu}$  follows from the fact that $\St_{r}$ is injective over $G_r$, and so $\Ext^i_{G_r}(k,\St_r\otimes\nabla(\mu)\otimes M) = 0$ for $i > 0$. 
We will call these functors {\em generalized Frobenius contraction functors}. When $\mu=(p^r-1)\rho$ these 
functors were introduced by Gros and Kaneda \cite{GK17} and later investigated by Andersen \cite{andersen17}. 

\subsection{} The following theorem demonstrates that the functor ${\mathcal F}_{\mu}$ can be expressed in terms of 
induction from $B/B_{r}$ to $G/G_{r}$.  Note that (for $\mu \in X_+$ and a $G$-module $M$), as $B/B_r$-modules,
$$
\Hom_{B_r}(k,\St_{r}\otimes \mu \otimes M) \cong \Hom_{U_r}(k,\St_{r}\otimes\mu\otimes M)^{T_r} \cong 
[\mu-(p^{r}-1)\rho \otimes M]^{T_{r}},
$$
as $\St_r^{U_r} = -(p^r - 1)\rho$ (cf. \cite[II.10.11]{rags}).

\begin{theorem}\label{T:functorviainduction} Let $\mu\in X_{+}$. Then 
\begin{itemize} 
\item[(a)] ${\mathcal F}_{\mu}(M)=\operatorname{Hom}_{G_{r}}(k,\operatorname{ind}_{B}^{G} (\St_{r}\otimes \mu\otimes M))
\cong \operatorname{ind}_{B/B_{r}}^{G/G_{r}}\left([\mu-(p^{r}-1)\rho \otimes M]^{T_{r}}\right)$.  
\item[(b)] $R^{i}\operatorname{ind}_{B/B_{r}}^{G/G_{r}}\left([\mu-(p^{r}-1)\rho \otimes M]^{T_{r}}\right)=0$ for $i>0$. 
\end{itemize} 
The $B/B_{r}$-structures are given by the isomorphism 
$$[\mu-(p^{r}-1)\rho \otimes M]^{T_{r}}\cong \operatorname{Hom}_{B_{r}}(k,\St_{r}\otimes \mu \otimes M),$$ 
\end{theorem} 

\begin{proof} Consider the following isomorphic functors: 
$$ {\mathcal F}_{1}(-)=\left(\text{ind}_{B}^{G}(-)\right)^{G_{r}}$$
$$ {\mathcal F}_{2}(-)=\text{ind}_{B/B_{r}}^{G/G_{r}}\left((-)^{B_{r}}\right).$$ 
As each arises as a composition, we obtain two spectral sequences, whose abutments agree (since the functors are isomorphic, cf. \cite[Proposition 3.1]{andersen-jantzen84}):
\begin{equation*} 
\hat{E}_{2}^{i,j}=\text{Ext}^{i}_{G_{r}}(k,R^{j}\text{ind}_{B}^{G}(\St_{r}\otimes \mu \otimes M))\Rightarrow (R^{i+j}{\mathcal F}_{1})(\St_{r}\otimes \mu\otimes M)
\end{equation*} 
\begin{equation*} 
E_{2}^{i,j}=R^{i}\text{ind}^{G/G_{r}}_{B/B_{r}} \text{Ext}^{j}_{B_{r}}(\St_{r},\mu\otimes M)\Rightarrow (R^{i+j}{\mathcal F}_{2})(\St_{r}\otimes\mu\otimes M). 
\end{equation*} 
By the generalized tensor identity, $R^j\ind_{B}^{G}(\St_r\otimes\mu\otimes M) \cong \St_r\otimes R^j\ind_{B}^{G}(\mu\otimes M)$, which is injective over $G_r$, since $\St_{r}$ is.  Hence, the first spectral sequence collapses. Precisely, $\hat{E}_2^{i,j} = 0$ for $i > 0$.  One can then identify the abutment and 
combine this with the second spectral sequence to obtain 
\begin{equation*} 
E_{2}^{i,j}=R^{i}\text{ind}^{G/G_{r}}_{B/B_{r}} \text{Ext}^{j}_{B_{r}}(\St_{r},\mu\otimes M)\Rightarrow \text{Hom}_{G_{r}}(k,R^{i+j}\text{ind}_{B}^{G}(\St_{r}\otimes \mu\otimes M)).
\end{equation*} 
This spectral sequence collapses and yields
$$R^{i}\text{ind}^{G/G_{r}}_{B/B_{r}} \text{Hom}_{B_{r}}(\St_{r},\mu\otimes M)\cong \text{Hom}_{G_{r}}(k,R^{i}\text{ind}_{B}^{G}(\St_{r}\otimes \mu\otimes M)).$$ 
The statement of (a) follows by setting $i=0$. From the generalized tensor identity and Kempf's Vanishing Theorem, one has 
$$R^{i}\text{ind}_{B}^{G}(\St_{r}\otimes \mu\otimes M) \cong [R^{i}\text{ind}_{B}^{G}\ \mu]\otimes  \St_{r}\otimes M=0$$ 
when $i>0$, which proves (b). 
\end{proof} 

From Theorem \ref{T:functorviainduction}, it is interesting to note that, for any $\mu \in X_{r}$, the $B/B_{r}$-module
$$[\mu-(p^{r}-1)\rho \otimes M]^{T_{r}}\cong \operatorname{Hom}_{B_{r}}(k,\St_{r}\otimes \mu \otimes M)$$ 
is acyclic with respect to the induction functor $\text{ind}_{B^{(r)}}^{G^{(r)}}(-)$. 

\subsection{} Given $\mu \in X_r$, let $\mu_{(r)} := (p^r - 1)\rho - \mu \in X_r$.  Note that the correspondence $\mu$ with $\mu_{(r)}$ gives a bijection on $X_r$.  In particular, in Theorem \ref{T:cohom-criteria}(b), $\mu$ may be replaced with $\mu_{(r)}$.  For any $\mu\in X$, 
let $\text{pr}_{\mu}$ be the functor that sends a rational $G$-module to the component in the (linkage) block defined by $\mu$ 
(cf. \cite[II. 7.3]{rags}). The next result gives conditions using the projection and generalized Frobenius contraction functors to ensure that $\text{St}_{r}\otimes M$ has a good filtration. 

\begin{theorem}\label{T:Stblockcond} Let $M$ be a rational $G$-module with $\dim \operatorname{Hom}_{G}(\Delta^{(p,r)}(\lambda),M)< \infty$ for 
all $\lambda \in X_{+}$. 
\begin{itemize} 
\item[(a)] If $\operatorname{pr}_{(p^{r}-1)\rho}(L(\mu)\otimes M)$ has a good filtration for all $\mu\in X_{r}$, then 
$\St_{r}\otimes M$ has a good filtration. 
\item[(b)] If ${\mathcal F}_{\mu}(M)=\operatorname{ind}_{B/B_{r}}^{G/G_{r}}\left([\mu-(p^{r}-1)\rho \otimes M]^{T_{r}}\right)$  
has a good filtration for all $\mu\in X_{r}$, where the $B/B_{r}$-structure is given by the isomorphism 
$$[\mu-(p^{r}-1)\rho \otimes M]^{T_{r}}\cong \operatorname{Hom}_{B_{r}}(k,\St_{r}\otimes \mu \otimes M),$$ 
then $\St_{r}\otimes M$ has a good filtration. 
\end{itemize}
\end{theorem} 

\begin{proof} (a) Suppose that $\operatorname{pr}_{(p^{r}-1)\rho}(L(\mu)\otimes M)$ has a good filtration for all $\mu\in X_{r}$. 
Then, since $\St_r\otimes \Delta(\si)^{(r)} \cong \Delta((p^r-1)\rho + p^r\si)$ (dual to \cite[Proposition 3.19]{rags}),
$$\Ext_G^1(\St_r\otimes \Delta(\sigma)^{( r)}, L(\mu)\otimes M)=0$$ 
for all $\mu\in X_{r}$ and $\sigma\in X_{+}$. 
The Lyndon-Hochschild-Serre spectral sequence 
$$E^{i,j}_{2}=\Ext^{i}_{G/G_r}(\Delta(\sigma)^{( r)}, \Ext^{j}_{G_{r}}(\St_{r}\otimes L(\mu^{*}),M))\Rightarrow 
\Ext_G^{i+j}(\St_r\otimes \Delta(\sigma)^{( r)}, L(\mu)\otimes M)$$  
collapses because $\St_{r}$ is projective as a $G_{r}$-module and yields the isomorphism: 
$$\Ext_G^1(\St_r\otimes \Delta(\sigma)^{( r)}, L(\mu)\otimes M) \cong \Ext^{1}_{G/G_{r}}(\Delta(\sigma)^{(r )}, \Hom_{G_{r}}(\St_{r}\otimes L(\mu^{*}),M)).$$ 
Therefore, $\Hom_{G_{r}}(\St_{r}\otimes L(\mu^{*}),M))$
has a good filtration as a $G/G_{r}$-module. The $G/G_{r}$-module $\Hom_{G_{r}}(T(\hat{\mu}_{(r)}), M)$ is a direct summand of 
$\Hom_{G_{r}}(\St_{r}\otimes L(\mu^{*}), M)$, because $T(\hat{\mu}_{(r)})$ is a $G$-direct summand 
of $\St_{r}\otimes L(\mu^{*})$ by \cite[Section 2, Corollary A]{pillen93}. Therefore, $\Hom_{G_{r}}(T(\hat{\mu}_{(r)}), M)$ has a $G/G_{r}$-good filtration. It follows by 
Theorem~\ref{T:cohom-criteria} that $\St_{r}\otimes M$ has a good filtration. 

(b) By duality and Theorem \ref{T:functorviainduction}, we have
$$
\Hom_{G_{r}}(\St_{r}\otimes \Delta(\mu^{*}), M) \cong \Hom_{G_r}(k,\St_r\otimes\nabla(\mu)\otimes M) \cong  \text{ind}_{B/B_{r}}^{G/G_{r}}\left([\mu-(p^{r}-1)\rho \otimes M]^{T_{r}}\right).
$$
Following \cite[Section 2, Corollary A]{pillen93} (see \cite[Remark 4.1.4]{So}) , $T(\hat{\mu}_{(r)})$ is also a $G$-direct summand of $\St_{r}\otimes \Delta(\mu^{*})$ for $\mu \in X_r$. 
Therefore, if 
$\text{ind}_{B/B_{r}}^{G/G_{r}}\left([\mu-(p^{r}-1)\rho, \otimes M]^{T_{r}}\right)$ has a good filtration as a $G/G_{r}$-module for all $\mu\in X_{r}$, then, 
by Theorem~\ref{T:cohom-criteria}, $\St_{r}\otimes M$ has a good filtration. 
\end{proof} 

Note that  in part  (a) of the previous theorem the module  $L(\mu)$ could be replaced by any of the following: $\nabla(\mu)$, $\Delta(\mu)$ or $T(\mu)$.


\section{Applications: Tensoring with simple modules} 

\subsection{} In this section we will apply the results from the previous section to verify cases when 
$\St_{r}\otimes L(\la)$ has a good filtration. In order to do so, the following result of Kildetoft and Nakano shows that it suffices to focus on the case when $r=1$. 
 
\begin{theorem}\label{T:reduces to r=1} {\cite[Prop. 4.4.1]{KN}} If $\St_{1}\otimes L(\la)$ has a good filtration for all $\la\in X_{1}(T)$, then 
$\St_{r}\otimes L(\la)$ has a good filtration for all $\la\in X_{r}(T)$, $r\geq 1$. 
\end{theorem}

\subsection{General bound on $p$} In \cite[Thm. 5.3.1, Thm. 9.4.1]{KN} \cite[Prop. 2.1]{andersen01}, it was shown that 
$\St_{r}\otimes L(\la)$ has a good filtration for $p\geq 2h-2$. This bound agrees 
with the present state of Donkin's Tilting Module Conjecture. The following result lowers 
the general bound for Conjecture~\ref{donkinconj}($\Rightarrow)$ to hold. This fact will be later used for 
our analysis of low rank groups. 

\begin{theorem} \label{T: boundonp} Let $\la \in X_{r}(T)$ and $p\geq 2h-4$. Then $\St_{r}\otimes L(\la)$ has a good filtration. 
\end{theorem} 

\begin{proof}
By Theorem \ref{T:reduces to r=1}, it suffices to prove this for $r=1$, and we will do so by using the characterization given in Theorem \ref{T:Stblockcond}(a) applied to $M = L(\la)$.  Suppose 
that $\mu, \lambda \in X_1(T)$, and that $\text{St}_{1} \otimes L(\gamma)^{(1)}$ is a composition factor of $L(\mu) \otimes L(\la)$.  Since $(p-1)\rho + p\gamma \le \mu + \lambda$, 
\begin{eqnarray*}
\langle (p-1)\rho + p\gamma, \alpha_0^{\vee} \rangle & \le & \langle \mu + \lambda, \alpha_0^{\vee} \rangle\\
& \le & 2(p-1)(h-1).
\end{eqnarray*} 
Thus
$$\langle \gamma, \alpha_0^{\vee} \rangle \le \frac{(p-1)}{p}(h-1) < h-1.$$
From this, we have
$$\langle \gamma + \rho, \alpha_0^{\vee} \rangle < 2(h-1),$$
so that
$$\langle \gamma + \rho, \alpha_0^{\vee} \rangle \le 2h-3.$$
If $p \ge 2h-3$, then $\gamma$ is contained in the closure of the fundamental alcove, and $L(\gamma) \cong \nabla(\gamma)$.  This proves the result for $p \ge 2h-3$.  The case when $p=2h-4$ only occurs if $p=2$ and $h=3$.  But this result (indeed, the Tilting Module Conjecture as well) is known to hold for $SL_3$ in characteristic $2$, as the four restricted simple modules are all tilting in this case.  Therefore, the result holds when $p \ge 2h-4$.
\end{proof}

\subsection{General bound on $\la$} One can also give a general upper bound on $\lambda$ that will ensure that tensoring the 
$r$th-Steinberg with a simple $G_{r}$-module will have a good filtration. 

\begin{proposition}\label{P: boundonweight}
If $\la \in X(T)_+$ and $\langle \la, \alpha_0^{\vee} \rangle \le 2p^r-1$, then $\St_{r}\otimes L(\la)$ has a good filtration. 
\end{proposition}

\begin{proof}
We work again with the characterization in Theorem \ref{T:Stblockcond}(a).  Suppose that $\St_r \otimes L(\gamma)^{(r)}$ is a composition factor of $L(\mu) \otimes L(\lambda)$ for some $\mu \in X_r(T)$.  
One has 
$$\langle \mu, \alpha_0^{\vee} \rangle \le \langle (p^r-1)\rho, \alpha_0^{\vee} \rangle,$$
so it follows that
$$p^r\langle \gamma, \alpha_0^{\vee} \rangle  \le \langle \la, \alpha_0^{\vee} \rangle \le 2p^r-1.$$
Thus
$$\langle \gamma, \alpha_0^{\vee} \rangle \le 2 - 1/p^r<2,$$
forcing $\gamma$ to be minuscule and therefore $L(\gamma) \cong \nabla(\gamma)$.
\end{proof}

\subsection{Rank 2 groups} The following theorem completes work on rank two groups initiated in 
\cite[Section 8]{KN}.

\begin{theorem} Assume the Lie rank of $G$ is less than or equal to two. Then 
$\St_{r}\otimes L(\la)$ has a good filtration for all $\la \in X_{r}(T)$. 
\end{theorem} 

\begin{proof}
Again, we use Theorem \ref{T:reduces to r=1} to reduce to the case that $r=1$.  In \cite[8.5]{KN}, the claim was shown in all cases except for when $\Phi$ is of type $G_2$ and $p = 7$, so we consider this case.  
Here $h = 6$ and $\alpha_0^{\vee} = 2\alpha_1^{\vee} + 3\alpha_2^{\vee}$.   Set $\lambda = a\varpi_1 + b\varpi_2$ with $0 \leq a, b \leq 6$.  We may assume that $\lambda \neq 6\rho$, so at least one of $a$ or $b$ is strictly less than 6.

Suppose now that $\mu \in X_1(T)$, and that $\St \otimes L(\gamma)^{(1)}$ is a composition factor of $L(\mu) \otimes L(\lambda)$.  Then
$$7\langle \gamma, \alpha_0^{\vee} \rangle \le 2a+3b \le 10+18=28,$$
with equality occurring only if $\lambda=5\varpi_1+6\varpi_2$.  
So $\langle \gamma,\al_0^{\vee}\rangle \leq 4$.  
Let $\gamma=c\varpi_1 + d\varpi_2$ with $0\le c,d \le 6$.  Then $2c+3d = \langle\ga,\al_0^{\vee}\rangle\le 4$. So $c \leq 2$, $d \leq 1$, and at most one of $c$ or $d$ may be non-zero.  
\vskip .25cm 

\noindent {\it Case I.} $\gamma = 0$: Then $L(\gamma) = L(0) = k = \nabla(0)$.

\bigskip
\noindent{\it Case II.} $\gamma = \varpi_1$: This lies in closure of the bottom alcove, and so $L(\varpi_1) = \nabla(\varpi_1)$.

\bigskip
\noindent{\it Case III.} $\gamma = \varpi_2$: Similarly, this does not lie in the bottom alcove, however, there is nothing lower linked to it, and so $L(\varpi_2) = \nabla(\varpi_2)$.

\bigskip
\noindent{\it Case IV.} $\gamma = 2\varpi_1 = 2\alpha_0$: Note that this is the same weight observed in \cite[8.5.4]{KN} to be problematic.  Here $L(2\varpi_1) \neq \nabla(2\varpi_1)$.  This situation 
occurs only if $\mu = (p-1)\rho$ and $\lambda = 5\varpi_1 + 6\varpi_2$.   But
$$\Hom_{G_1}(\St_{1}, \St_{1} \otimes L(\lambda))^{(-1)} \cong \Hom_{G_1}(\St_{1} \otimes \St_{1}, L(\lambda))^{(-1)},$$
and it suffices to only check this for the tilting summand of highest weight in $\text{St}_{1}\otimes \text{St}_{1}$.  But $12\rho$ is not $G_1$-linked to $\lambda$, so $\Hom_{G_1}(T(12\rho)),L(\lambda))^{(-1)}=0$.
\end{proof}

Although we rely on \cite[Section 8.2]{KN} to remove most of the cases, the results in this paper could have been used in other type $G_2$ cases and lead to very short proofs for the other rank $\le 2$ groups.  For example, Theorem \ref{T: boundonp} (and its proof) establish the result for $SL_2$ and $SL_3$ in all characteristics.  For type $B_2$, we have $h=4$, so that the result holds for all $p \ge 4$ by Theorem \ref{T: boundonp}, leaving only $p=2,3$ to check.  If $p=2$ and $\la \in X_1(T)$, then $\langle \la, \alpha_0^{\vee} \rangle \le 3 = 2p-1$.  If $p=3$ and $\la \in X_1(T)$ is not the Steinberg weight (for which the result is clear), then $\langle \la, \alpha_0^{\vee} \rangle \le 5 = 2p-1$.  Thus, in both of these cases, the result holds by Proposition \ref{P: boundonweight}.

\subsection{Fundamental Weights} We now consider the case of a restricted irreducible $G$-module 
where the highest weight is a fundamental weight.  To do this, we need to extend the usual partial order on weights to a parital ordering $\leq_{\mathbb{Q}}$ relative to the rational numbers.  For $\mu, \la \in X$, we say that $\mu \leq_{\mathbb{Q}} \la$ if 
$$
\la - \mu = \sum_{\al \in S}q_{\al}\al
$$
for $q_{\al} \in \mathbb{Q}$ with $q_{\al} \geq 0$. Note that $\mu \leq \la$ implies $\mu \leq_{\mathbb{Q}} \la$.

\begin{theorem} Let the Lie rank of $G$ be $n$. Then 
\begin{itemize} 
\item[(a)] $\St_{r}\otimes L(\varpi_{j})$ has a good filtration for $j=1,2,\dots,n$ and $r\geq 2$. 
\item[(b)] $\St_{1}\otimes L(\varpi_{j})$ has a good filtration for $j=1,2,\dots,n$, except possibly when $p = 2$ and 
$\Phi=E_{7},\ E_{8}$. 
\item[(c)] In the case when $p=2$ and $\Phi=E_{7}$ or $E_{8}$, 
\begin{itemize} 
\item[(i)] $\St_{1}\otimes L(\varpi_{j})$ has a good filtration for $j\neq 4$ when $\Phi=E_{7}$
\item[(ii)] $\St_{1}\otimes L(\varpi_{j})$ has a good filtration for $j\neq  4,5$ when $\Phi=E_{8}$.  
\end{itemize}  
\end{itemize} 
\end{theorem} 

\begin{proof} We want to consider $\operatorname{ind}_{B/B_{r}}^{G/G_{r}}\left([\mu-(p^{r}-1)\rho \otimes L(\varpi_{j})]^{T_{r}}\right)$ 
for $\mu\in X_{r}(T)$ and show it has a good filtration. This will occur if all the dominant weights in 
$[\mu-(p^{r}-1)\rho \otimes L(\varpi_{j})]^{T_{r}}$ are of the form $p^{r}\delta$ with the property that  $H^{0}(\delta)=L(\delta)$. 

Let $p^{r}\delta$ be a dominant weight of $[\mu-(p^{r}-1)\rho \otimes L(\varpi_{j})]^{T_{r}}$. Then 
$p^{r}\delta=\mu-(p^{r}-1)\rho+\gamma$ where $\gamma$ is a weight of $L(\varpi_{j})$. Since 
$\mu\leq_{\mathbb Q} (p^{r}-1)\rho$, it follows that 
\begin{equation}
p^{r}\delta\leq_{\mathbb Q} \gamma.
\end{equation} 
By taking the inner product with $\alpha_{0}^{\vee}$, one obtains 
\begin{equation} 
0\leq \langle \delta,\alpha_{0}^{\vee} \rangle \leq \frac{1}{p^{r}}\langle \varpi_{j},\alpha_{0}^{\vee} \rangle. 
\end{equation} 
Set $h(j,r,p)=\frac{1}{p^{r}}\langle \varpi_{j},\alpha_{0}^{\vee}\rangle $. 

For types $A_{n}$, $B_{n}$, $C_{n}$, $D_{n}$, $E_{6}$ and $G_{2}$, one has 
$h(j,r,p)<2$ for all $j,r$. In these cases this implies that $\delta$ is either zero or minuscule 
and $H^{0}(\delta)=L(\delta)$. 

For type $F_{4}$, one can repeat this argument, but replacing $\alpha^{\vee}_0$ with the coroot of the highest root $\alpha_{h}^{\vee}=
2\alpha_{1}^{\vee}+3\alpha_{2}^{\vee}+2\alpha_{3}^{\vee}+\alpha_{4}^{\vee}$. In this case, one obtains 
$h(j,r,p)=\frac{1}{p^{r}}\langle \varpi_{j},\alpha_{h}^{\vee}\rangle<2$. This implies that $\delta=0$ or $\omega_{4}$ (with $p=2,3$). 
The only solution to $p^{r}\omega_{4}\leq_{\mathbb Q} \omega_{j}$ occurs when $p=2$, $r=1$ and $j=2$. However, 
$H^{0}(\omega_{4})=L(\omega_{4})$ when $p=2$ (cf. \cite[p. 299]{Jan91}). 

In the case when $\Phi=E_{7}$, one has $h(j,r,p)< 2$ unless $j=4$, $r=1$ and $p=2$. For 
$\Phi=E_{8}$, one has $h(j,r,p)< 2$ unless 
\begin{itemize} 
\item[(i)] $j=\varpi_{3}$, $r=1$, $p=2$;
\item[(ii)] $j=\varpi_{6}$, $r=1$, $p=2$;
\item[(iii)] $j=\varpi_{5}$, $r=1$, $p=2$;
\item[(iv)] $j=\varpi_{4}$, $r=1$, $p=2$;
\item[(v)] $j=\varpi_{4}$, $r=1$, $p=3$.
\end{itemize}
For type $E_{8}$, the root lattice and the weight lattice coincide, so in this case 
$$p^{r}\delta\leq \gamma \leq \varpi_{j}.$$
Suppose $\delta\neq 0$. Using \cite[Figure 3]{UGAVIGRE6}, in cases (i), (ii), (v) it follows that $\delta=\varpi_{8}$, and 
one has $H^{0}(\delta)=L(\delta)$ by \cite[Theorem 1.1]{GGN}. One is left with 
cases (iii), (iv) for type $E_{8}$. 
\end{proof}


\section{Higher Rank Cases} 

\subsection{} In this section, we consider the question of whether $\St_r\otimes L(\la)$ has a good filtration for some higher rank groups over small primes.  Here the results are less complete, and we will focus on the $r = 1$ situation.  For those cases where we can show that
\begin{equation}\label{claim} 
\St_1\otimes L(\la)\ \text{has a good filtration for all $\la \in X_1$},
\end{equation} 
it will then follow from Theorem \ref{T:reduces to r=1} that $\St_r\otimes L(\la)$ has a good filtration for all $\la \in X_r$ for $r > 1$.

Note that $L((p-1)\rho) \simeq \St_1$, and the claim (\ref{claim}) always holds for this particular weight. For a given $\la \in X_1$, we may therefore assume throughout that $\la \neq (p-1)\rho$.  We again make use of 
Theorem \ref{T:Stblockcond}(a).   Suppose that $\St_1\otimes L(\ga)^{(1)}$ is a composition factor of $L(\mu)\otimes L(\la)$ for $\ga \in X_+$ and $\mu \in X_1$.  Our goal is to show that either no such $\ga$ and $\mu$ exist or that $L(\ga) = \nabla(\ga)$.   Recall the notation $\mu_{(1)} := (p-1)\rho - \mu$.  As in the proof of Theorem \ref{T: boundonp}, we must have
\begin{equation}\label{E:weightineq}
p\ga \leq \la - \mu_{(1)},
\end{equation}
from which we may conclude that
\begin{equation}\label{E:innerineq}
p\langle\ga,\al_0^{\vee}\rangle \leq \langle \la,\al_0^{\vee}\rangle \hskip0.5in \text{and} \hskip0.5in
p\langle\ga,\ta^{\vee}\rangle \leq \langle \la,\ta^{\vee}\rangle,
\end{equation}
where $\ta$ denotes the highest root.  These inequalities are often sufficient to eliminate options, but further reductions can also be made by noting that $\la$ must be $G_1$-linked to $\mu_{(1)}$.  

\subsection{Rank 3 groups.} For rank three groups, the claim (\ref{claim}) holds in almost all cases (cf. also \cite[\S 8.3]{KN}).

\begin{theorem} Let $G$ be of type $A_3$, $B_3$, or $C_3$ and $\la \in X_1$.  Then $\St_1\otimes L(\la)$ has a good filtration except possibly for the following cases:
\begin{itemize}
\item Type $B_3$ with $p = 7$: $\la =(6,5,5), (6,4,5), (6,5,4), (5,5,5), (5,5,4), (5,5,4)$, $(5,4,5)$, $(4,5,5)$, $(4,5,4)$, or $(3,5,5)$,
\item Type $C_3$ with $p = 3$: $\la = (2,1,2) \text{ or } (2,2,1)$,
\item Type $C_3$ with $p = 7$: $\la = (6,5,5), (6,4,5), (6,5,4), (5,5,5), \text{ or } (4,5,5)$.
\end{itemize}
\end{theorem}

\begin{proof} The case of type $A_3$ was shown by Kildetoft and Nakano \cite[\S 8.3]{KN}, but also follows from our previous results. By Theorem \ref{T: boundonp}, we are reduced to $p = 2$ or 3.  But those cases follow from Proposition \ref{P: boundonweight}.

For type $B_3$, $h = 6$.  By Theorem \ref{T: boundonp}, we are done if $p > 7$.  Note that $\al_0^{\vee} = 2\al_1^{\vee} + 2\al_2^{\vee} + \al_3^{\vee}$ and $\ta^{\vee} = \al_1^{\vee} + 2\al_2^{\vee} + \al_3^{\vee}$, from which we see that $\langle\rho,\al_0^{\vee}\rangle = 5$ and $\langle\rho,\ta^{\vee}\rangle = 4$.

For $p = 2$, Proposition \ref{P: boundonweight} reduces us to $\la = \varpi_1 + \varpi_2$.  The only $\ga$ and $\mu$ satisfying \eqref{E:weightineq} are $\ga = \varpi_1$ and $\mu = \rho$ (so $L(\mu) = \St_1$).  
Since $\St_1\otimes L(\varpi_1+\varpi_2) \subset \St_1\otimes\nabla(\varpi_1 + \varpi_2)$, 
if $\St_1\otimes L(\varpi_1)^{(1)}$ is a composition factor of $\St_1\otimes L(\varpi_1+\varpi_2)$, then it is also a composition factor of $\St_1\otimes\nabla(\varpi_1 + \varpi_2)$.   But, that is not the case based on explicit weight computations using the software LiE \cite{LiE}. 

For $p = 3$, applying both inequalities in \eqref{E:innerineq}, we are reduced to the following options for $\ga$: $\varpi_1$, $\varpi_2$, $\varpi_3$, $2\varpi_3$, or $\varpi_1 + \varpi_3$. Applying \eqref{E:weightineq} to $\ga = 2\varpi_3$, we would need $6\varpi_3 \leq \lambda - \mu_{(1)} <_{\mathbb{Q}} 2\rho$, which fails to hold.  That leaves us with $\ga = \varpi_i$ or $\varpi_1 + \varpi_3$.  For $p \neq 2$, each $\nabla(\varpi_i)$ is simple (cf. \cite[II.8.21]{rags}).  Further, $\nabla(\varpi_1 + \varpi_3)$ is also known to be simple, except when $p = 7$ (cf. \cite{GGN}). 

For $p = 5$, \eqref{E:innerineq} reduces us to $\ga = \varpi_1$, $\varpi_2$, $\varpi_3$, $2\varpi_3$, $3\varpi_3$, $\varpi_1 + \varpi_3$, or $\varpi_2 + \varpi_3$.   From \eqref{E:weightineq}, one has $5\ga \leq \la - \delta <_{\mathbb{Q}} 4\rho$.  This fails to hold for $\ga = 3\varpi_3$.   While it is true that $5(\varpi_2 + \varpi_3) <_{\mathbb{Q}} 4\rho$, there is no $\la \neq 4\rho$ with $5(\varpi_2 + \varpi_3) \leq \la - \mu_{(1)}$.   So this reduces us to $\ga = \varpi_i$, $\varpi_1 + \varpi_3$, or $2\varpi_3$.  As noted above, each $\nabla(\varpi_i)$ and $\nabla(\varpi_1 + \varpi_3)$ is simple.  Lastly, by explicit dimension computations of L{\"u}beck \cite{L}, $\nabla(2\varpi_3)$ is simple for all odd primes. 

For $p = 7$, from \eqref{E:innerineq}, we are reduced to the following options for $\ga$: 
$\varpi_1$, $\varpi_2$, $\varpi_3$, $2\varpi_1$, $2\varpi_3$, $3\varpi_3$, $\varpi_1 + \varpi_2$, $\varpi_1 + \varpi_3$, $\varpi_2 + \varpi_3$, $\varpi_1 + 2\varpi_3$.   Using known facts and dimension computations of L{\"u}beck \cite{L}, the only cases where $\nabla(\ga)$ is not simple are $\ga = 2\varpi_1$ or $\varpi_1 + \varpi_3$.  Both can satisfy \eqref{E:weightineq}.  In particular, for $\ga = 2\varpi_1$, we can have $\la - \mu_{(1)} = 6\rho - \varpi_1$, $6\rho - \varpi_2$, or $6\rho - 2\varpi_3$.  One can check that, in each case, $\la$ is not $G_1$-linked to $\mu_{(1)}$.  So this leaves only the second case of $\ga = \varpi_1 + \varpi_3$, which admits a large number of options for $\la$ (and $\mu_{(1)}$).  However $G_1$-linkage holds only in the following cases:

\vskip.1in
\begin{tabular}{|c|c|c|c|c|c|c|c|c|c|c|c|}
\hline
$\la$ & (6,5,5) & (6,4,5) & (6,5,4) & (5,5,5) & (5,5,4) & (5,5,4) &(5,5,5) & (5,4,5) & (4,5,5) & (4,5,4) & (3,5,5)\\
\hline
$\mu_{(1)}$ & (3,0,0) & (1,1,0) & (2,0,1) & (0,0,0) & (0,0,1) & (1,0,1) & (2,0,0)& (0,1,0) & (1,0,0) & (0,0,1) & (0,0,0)\\
\hline
\end{tabular}

\vskip.1in
For type $C_3$, we are again done if $p > 7$.  In this case $\al_0^{\vee} = \al_1^{\vee} + 2\al_2^{\vee} + 2\al_3^{\vee}$ and $\ta^{\vee} = \al_1^{\vee} + \al_2^{\vee} + \al_3^{\vee}$, from which we see that $\langle\rho,\al_0^{\vee}\rangle = 5$ and $\langle\rho,\ta^{\vee}\rangle = 3$.
 
For $p = 2$, Proposition \ref{P: boundonweight} reduces us to $\la = \varpi_2 + \varpi_3$.  
The only weight $\ga$ satisfying \eqref{E:weightineq} is $\ga = \varpi_1$, which is miniscule. Hence, $\nabla(\varpi_1)$ is simple.

For $p = 3$, \eqref{E:innerineq} reduces us to $\ga = \varpi_i$ for some $1 \leq i \leq 3$.  Again, $\nabla(\varpi_1)$ is simple.  Premet and Suprunenko \cite{PrSu} showed that $\nabla(\varpi_2)$ is simple if and only if $p \neq 3$ and $\nabla(\varpi_3)$ is simple if and only if $p \neq 2$.  So we are reduced to the case $\ga = \varpi_2$, which does satisfy \eqref{E:weightineq} for many values of $\la - \mu_{(1)}$. The following table summarizes the possible cases where $\la$ is $G_1$-linked to $\mu_{(1)}$.

\vskip.1in
\begin{tabular}{|c|c|c|c|c|}
\hline
$\la$ & (2,1,2) & (2,1,2) & (2,2,1) & (2,0,2)\\
\hline
$\mu_{(1)}$ & (0,0,0) & (0,1,0) & (1,0,0) & (0,0,0)\\
\hline
\end{tabular}

\vskip.1in\noindent
With some further investigation, we can eliminate two of those four options.  For $\nu \in X_1$, set $Z_1(\nu) := \operatorname{coind}_{B_1^+T}^{G_1T}(\nu)$.  First, suppose that $\la = (2,0,2)$ and $\mu_{(1)} = (0,0,0)$ (the last case).  Then $\mu = (p-1)\rho$ and $L(\mu) \cong \St_1$.  We are reduced to showing that $\operatorname{pr}_{(p-1)\rho}(\St_1\otimes L(\la))$ has a good filtration or, equivalently, that $N:= \operatorname{Hom}_{G_1}(\St_1, \St_1\otimes L(\la))^{(-1)}$  has a good filtration. This can only fail if $L(\varpi_2)$ is a composition factor of $N$.  The weight-space multiplicity of $5 \varpi_2$ in $\St_1$ is one. It follows that $Z_1(2 \rho + 5 \varpi_2)=Z_1(2(p-1)\rho -\la + p\varpi_2)$ appears exactly once as a section of the $G_1T$-module $\St_1 \otimes \St_1$. Moreover, the only weights in $\St_1 \otimes \St_1$ that are higher than $2(p-1)\rho -\la + p\varpi_2$ are $4 \rho$, $4 \rho- \alpha_1$, and $4 \rho - \alpha_3.$ The weight $2(p-1)\rho -\la + p\varpi_2$ is not strongly linked to any of these. Hence,  $2(p-1)\rho -\la + p\varpi_2$ is a maximal weight inside $\St_1 \otimes \St_1$.  It follows that $Q_1(\la + p \varpi_2)$ appears exactly once as a summand of the $G_1T$-module $\St_1 \otimes \St_1$ and that $L(\varpi_2)$ appears exactly once as a composition factor of $N$.

By the same argument we can conclude that the induced module $\nabla (2(p-1)\rho -\la + p\varpi_2)$ also appears exactly once as a section in the good filtration of the $G$-module $\St_1 \otimes \St_1$.  The fact that $2(p-1)\rho -\la + p\varpi_2$ is a maximal weight in $\St_1\otimes\St_1$ says that the tilting module $T(2(p-1)\rho -\la + p\varpi_2) \cong T(2(p-1)\rho -\la) \otimes  T(\varpi_2)^{(1)}$ is a summand of the tilting module $\St_1 \otimes \St_1$. This implies that the tilting module $T(\varpi_2)$ is a summand of $N$. Since the multiplicity of $L(\varpi_2)$ in $N$ is one,  it appears inside this summand. All other composition factors $L(\eta)$ of $N$ satisfy $L(\eta) \cong \nabla (\eta)$. One concludes that $N$ is tilting, and thus has a good filtration, eliminating the weight $(2,0,2)$ from the above list.

Consider now the second case in the above list: $\la = (2,1,2)$ and $\mu_{(1)} = (0,1,0)$. Then $\mu = (p-1)\rho - \mu_{(1)} = \la$.  An argument similar to the preceding case shows that $\operatorname{pr}_{(p-1)\rho}(T(\la)\otimes L(\la))$ has a good filtration, which eliminates this case as well, leaving only the following unknown cases:

\vskip.1in
\begin{tabular}{|c|c|c|}
\hline
$\la$ & (2,1,2)  & (2,2,1) \\
\hline
$\mu_{(1)}$ & (0,0,0)  & (1,0,0) \\
\hline
\end{tabular}

\vskip.1in
For $p = 5$, \eqref{E:innerineq} reduces us to the following options for $\ga$: $\varpi_1$, $\varpi_2$, $\varpi_3$, $2\varpi_1$, $\varpi_1 + \varpi_2$, or $\varpi_1 + \varpi_3$.  Here we know the simplicity of $\nabla(\varpi_i)$ for each $i$.  From dimension computations of L{\"u}beck \cite{L}, $\nabla(2\varpi_1)$ is simple for $p > 2$, $\nabla(\varpi_1 + \varpi_2)$ is simple if and only if $p \neq 3$ or 7, and $\nabla(\varpi_1 + \varpi_3)$ is simple if and only if $p > 3$.  In particular, all are simple for $p = 5$.

For $p = 7$, \eqref{E:innerineq} reduces us to the following options for $\ga$: $\varpi_1$, $\varpi_2$, $\varpi_3$, $2\varpi_1$, $2\varpi_2$, $2\varpi_3$, $\varpi_1 + \varpi_2$, $\varpi_1 + \varpi_3$, $\varpi_2 + \varpi_3$.  From previous discussions, the only cases where $\nabla(\ga)$ is not simple are $\ga = 2\varpi_2, \varpi_1 + \varpi_2$.  Both can satisfy \eqref{E:weightineq}.  For $\ga = 2\varpi_2$, there are three options for $\la: 6\varpi_1 + 5\varpi_2 + 6\varpi_3$ (with $\mu_{(1)} = 0$); $4\varpi_1 + 6\varpi_2 + 6\varpi_3$ (with $\mu_{(1)} = 0$); and $5\varpi_1 + 6\varpi_2 + 6\varpi_3$ (with $\mu_{(1)} = \varpi_1)$.  However, one can directly check that in each case $\la$ and $\mu_{(1)}$ are not $G_1$-linked.  In the second case ($\ga = \varpi_1 + \varpi_2$), there are numerous values of $\la$ that satisfy \eqref{E:weightineq}, however $G_1$-linkage holds only in the following cases:

\vskip.1in
\begin{tabular}{|c|c|c|c|c|c|}
\hline
$\la$ & (6,5,5) & (6,4,5) & (6,5,4) & (5,5,5) & (4,5,5)\\
\hline
$\mu_{(1)}$ & (0,0,0) & (0,1,0) & (0,0,1) & (1,0,0) & (0,0,0)\\
\hline
\end{tabular}

\end{proof}

\subsection{Rank 4 groups.} In types $A_4$ and $D_4$, the claim also holds in almost all cases. While potentially problematic weights are not listed explicitly in the following theorem, some information is provided in the proof. 

\begin{theorem} Assume $G$ is of type $A_4$ or $D_4$ and $\la \in X_1$.  Then $\St_1\otimes L(\la)$ has a good filtration except possibly for the following cases:
\begin{itemize}
\item Type $A_4$ with $p = 5$,
\item Type $D_4$ with $p = 7$.
\end{itemize}
\end{theorem}

\begin{proof} We first consider type $A_4$, where $h = 5$.  By Theorem \ref{T: boundonp}, we are done for $p > 5$.  For $p = 2$, the result follows from Proposition \ref{P: boundonweight}.

For $p = 3$, one could again eliminate many $\la$ via Proposition \ref{P: boundonweight}. However, we more directly focus on the weight $\ga$. First, \eqref{E:innerineq} reduces us to $\ga = \varpi_i + \varpi_j$ for $i, j \in \{1, 2, 3, 4\}$.  Of those, the only weights potentially satisfying \eqref{E:weightineq} are $\ga = \varpi_1 + \varpi_3$, $\varpi_2 + \varpi_4$, or $\varpi_1 + \varpi_4$.  However, in each case $\nabla(\ga)$ is simple, as can be seen by using Jantzen's algorithm \cite[Satz 9]{Jan73} (cf. also \cite[II.8.21]{rags}) for checking the simplicity of a standard induced module in type $A_n$.

For $p = 5$, using \eqref{E:innerineq}, \eqref{E:weightineq}, and Jantzen's algorithm for simplicity, one can reduce the problem to just one possible value of $\ga$: $\varpi_1 + \varpi_4 = \al_{0}$.  We have the following values of $\la$ and $\mu_{(1)}$ which are $G_1$-linked and satisfy \eqref{E:weightineq}.

\vskip.1in
\begin{tabular}{|c|c|c|c|c|c|c|c|}
\hline
$\la$ & (4,3,3,4) & (4,3,2,4) & (4,2,3,4) & (3,3,2,4) & (3,2,3,4) & (4,3,2,3) & (4,2,3,3)\\
\hline
$\mu_{(1)}$ & (2,0,0,2) & (1,1,0,1) & (1,0,1,1) & (0,1,0,1) & (0,0,1,1) & (1,1,0,0) & (1,0,1,0)\\
\hline
\end{tabular}

\vskip0.1in
\begin{tabular}{|c|c|c|c|c|c|c|c|}
\hline
$\la$ & (4,3,3,3) & (3,3,3,4) & (2,3,3,4) & (4,3,3,2) & (4,3,1,4) & (4,1,3,4) & (4,2,2,4)\\  
\hline
$\mu_{(1)}$ & (2,0,0,1) & (1,0,0,2) & (0,0,0,2) & (2,0,0,0) & (0,2,0,0) & (0,0,2,0) & (0,1,1,0)\\
\hline
\end{tabular}

\vskip0.1in
\begin{tabular}{|c|c|c|c|c|c|c|c|}
\hline
$\la$ & (3,3,3,3) & (3,3,3,3) & (2,3,3,3) & (3,2,3,3) & (3,3,2,3) & (3,3,3,2) & (2,3,3,2)\\
\hline
$\mu_{(1)}$ & (1,0,0,1) & (0,0,0,0) & (0,0,0,1) & (0,0,1,0) & (0,1,0,0) & (1,0,0,0) & (0,0,0,0)\\
\hline
\end{tabular}

\vskip.1in
For type $D_4$, $h = 6$. By Theorem \ref{T: boundonp}, we are done if $p > 7$.  We have $\al_0^{\vee} = \ta^{\vee} = \al_1^{\vee} + 2\al_2^{\vee} + \al_3^{\vee} + \al_4^{\vee}$, and so $\langle\rho,\al_0^{\vee}\rangle = 5$.

For $p = 2$, using \eqref{E:innerineq} and \eqref{E:weightineq}, we are reduced to $\ga = \varpi_1$, $\varpi_2$, or $\varpi_3$. But all those weights are miniscule, giving a simple $\nabla(\ga)$.  

For $p = 3$, many values of $\ga$ satisfy \eqref{E:innerineq}.  Using dimension computations of L{\"u}beck \cite{L}, one finds that the only cases where $\nabla(\ga)$ is not simple are as follows: $3\varpi_1$, $3\varpi_3$, $3\varpi_4$, $\varpi_1 + \varpi_2$, $\varpi_2 + \varpi_3$, $\varpi_2 + \varpi_4$, and $\varpi_1 + \varpi_3 + \varpi_4$.  By direct verification, none of these can satisfy \eqref{E:weightineq}.

For $p = 5$, similarly, \eqref{E:innerineq} and dimension computations of L{\"u}beck \cite{L} reduce us to the following options for $\ga$: $3\varpi_i$, $i \in \{1, 3, 4\}$; $4\varpi_i$, $i \in \{1, 3, 4\}$; $2\varpi_i + \varpi_j$, $i, j \in \{1, 3, 4\}$, $i \neq j$; $3\varpi_i + \varpi_j$, $i, j \in \{1, 3, 4\}$, $i \neq j$; $2\varpi_i + 2\varpi_j$, $i, j \in \{1, 3, 4\}$, $i \neq j$; $\varpi_i + 2\varpi_2$, $i \in \{1, 3, 4\}$; and $\varpi_2 + \varpi_i + \varpi_j$, $i, j \in \{1, 3, 4\}$, $i \neq j$.  One then checks that \eqref{E:weightineq} fails to hold in all cases.

For $p = 7$, as above, \eqref{E:innerineq} and L{\"u}beck's computations reduce us to the following options for $\ga$: $\varpi_i + \varpi_2$, $i \in \{1, 3, 4\}$, $2\varpi_2$.  Unfortunately, \eqref{E:weightineq} can hold here.  In the case $\ga = 2\varpi_2$, the only values of $\la$ and $\mu_{(1)}$ that work are as follows:

\vskip0.1in
\begin{tabular}{|c|c|c|c|c|c|c|c|}
\hline
$\la$ & (6,5,6,6) & (4,6,6,6) & (6,6,4,6) & (6,6,6,4) & (5,6,6,6) & (6,6,5,6) & (6,6,6,5)\\
\hline
$\mu_{(1)}$ & (0,0,0,0) & (0,0,0,0) & (0,0,0,0) & (0,0,0,0) & (1,0,0,0) & (0,0,1,0) & (0,0,0,1)\\
\hline
\end{tabular}

\vskip0.1in\noindent
One can check that in each case $\la$ is not $G_1$-linked to $\mu_{(1)}$. So that case is eliminated.  

In the first (symmetric) cases for $\ga$, there are options for $\la$ and $\mu_{(1)}$ where linkage holds.  For example, for $\ga = \varpi_1 + \varpi_2$, one has the following cases where $G_1$-linkage holds between $\la$ and $\mu_{(1)}$:

\vskip.1in
\begin{tabular}{|c|c|c|c|c|}
\hline
$\la$ & (6,4,5,5) & (4,5,5,5) & (5,5,5,5) & (6,5,5,5) \\
\hline
$\mu_{(1)}$ & (0,1,0,0) & (0,0,0,0) & (1,0,0,0) & (2,0,0,0) \\
\hline
\end{tabular}

\vskip.1in\noindent
Similar cases would exist for $\mu_1 = \varpi_2 + \varpi_3$ and $\mu_1 = \varpi_2 + \varpi_4$.  Additional cases may also exist, as a complete list has not been computed.
\end{proof}

For types $B_4$ and $C_4$, $h = 8$, and the claim holds for $p > 7$.  No investigation has been made for small primes.  

For type $F_4$, $h = 12$, and we are done if $p > 17$.  We make some observations for $p = 2$. We have $\al_0^{\vee} = 2\al_1^{\vee} + 4\al_2^{\vee} + 3\al_3^{\vee} + 2\al_4^{\vee}$ and $\ta^{\vee} = 2\al_1^{\vee} + 3\al_2^{\vee} + 2\al_3^{\vee} + \al_4^{\vee}$.  So $\langle\rho,\al_0^{\vee}\rangle = 11$ and $\langle\rho,\ta^{\vee}\rangle = 8$. The inequalities in \eqref{E:innerineq} force $\langle\ga,\al_0^{\vee}\rangle \leq 5$ and $\langle\ga,\ta^{\vee}\rangle \leq 3$.  One has the following options: $\ga = \varpi_1$, $\varpi_2$, $\varpi_3$, $\varpi_4$, $2\varpi_4$, $\varpi_1 + \varpi_4$, or $\varpi_3 + \varpi_4$.  The case $\ga = \varpi_3 + \varpi_4$ may be eliminated as it does not satisfy \eqref{E:weightineq}, and the case case $\ga = \varpi_4$ may be eliminated as $\nabla(\varpi_4)$ is simple (for $p = 2$) by \cite{Jan91}.  One finds the following possibilities where \eqref{E:weightineq} holds and $\la$ is $G_1$-linked to $\mu_{(1)}$:

\vskip.1in
\begin{tabular}{|c|c|}
\hline
$\la$ & $\mu_{(1)}$\\
\hline
(1,1,0,1) & (0,0,0,0), (0,0,1,0), (0,0,0,1), (0,0,0,2)\\
\hline
(1,1,1,0) & (0,0,0,0), (1,0,0,0), (0,0,1,0), (0,0,0,1), (1,0,0,1)\\
\hline
(0,1,1,0) & (0,0,0,1)\\
\hline
(1,1,0,0) & (0,0,0,1)\\
\hline
\end{tabular}

\subsection{Type $A_5$.}\label{S:A5} For type $A_5$, $h = 6$.  By Theorem \ref{T: boundonp}, we are done for $p > 7$.  In contrast to the smaller rank cases, even for $p = 2$, our earlier methods do not completely resolve the issue.  Using \eqref{E:innerineq}, \eqref{E:weightineq}, and Jantzen's simplicity algorithm, one is reduced to one case: $\la = \varpi_1 + \varpi_2 + \varpi_4 + \varpi_4$ with $\mu_{(1)} = 0$ and $\ga = \varpi_1 + \varpi_5$.    Note that $\la$ is indeed $G_1$-linked to the zero weight.  However, through an intricate analysis of the modules involved, in Section \ref{S:char2}, we are able to address this case.  See Theorem \ref{T:typeA5}. For $p = 3$, 5, or 7, there will be many more options for $\la$ that cannot be dealt with by the above methods.

\subsection{Summary.} For $\la \in X_1$, $\St_1\otimes L(\la)$ has a good filtration in the following cases:

\begin{itemize}
\item Type $A_n$: $p > 2n - 3$
\subitem $\bullet$ Type $A_2$: all primes
\subitem $\bullet$ Type $A_3$: all primes 
\subitem $\bullet$ Type $A_4$: $p \neq 5$
\subitem $\bullet$ Type $A_5$: $p \neq 3, 5, 7$
\item Type $B_n$: $p > 4n - 5$
\subitem $\bullet$ Type $B_2$ (equivalently $C_2$): all primes
\subitem $\bullet$ Type $B_3$: $p \neq 7$
\subsubitem $\bullet$ $p = 7$ case, all except
$\la = (6,5,5), (6,4,5), (6,5,4), (5,5,5), (5,5,4), (5,5,4)$, 
\subsubitem $(5,4,5), (4,5,5), (4,5,4), \text{ or } (3,5,5)$
\item Type $C_n$: $p > 4n - 5$
\subitem $\bullet$ Type $C_3$: $p \neq 3, 7$
\subsubitem $\bullet$ $p = 3$ case, all except $\la = (2,1,2)$ or $(2,2,1)$
\subsubitem $\bullet$ $p = 7$ case, all except $\la = (6,5,5), (6,4,5), (6,5,4), (5,5,5), \text{ or } (4,5,5)$
\item Type $D_n$: $p \geq 4n - 9$
\subitem $\bullet$ Type $D_4$: $p \neq 7$
\item Type $E_6$: $p > 19$
\item Type $E_7$: $p > 31$
\item Type $E_8$: $p > 53$
\item Type $F_4$: $p > 19$
\subitem $\bullet$ $p = 2$ case, reduced to $\la = (1,1,0,1), (1,1,1,0), (0,1,1,0) \text{ or } (1,1,0,0)$
\item Type $G_2$: all primes
\end{itemize}

\section{A Detailed Analysis In Characteristic $2$}\label{S:char2}

In this section we investigate two very similar situations in which a proof that $\St_1 \otimes L(\lambda)$ has a good filtration is beyond the reach of our earlier arguments.  In particular, basic weight combinatorics are not conclusive, and thus it becomes necessary to better understand the submodule structure of a tensor product of $G$-modules.  We show in one of the two cases that we are able to verify that $\St_1\otimes L(\la)$ does have a good filtration, which allows us to conclude that Conjecture \ref{donkinconj} ($\Rightarrow$) holds.  That this holds in such a nontrivial setting could be viewed as the strongest evidence yet for its truth in arbitrary characteristic.  However, if this is indeed true, one will need to find the underlying reason why it holds in situations similar to those considered here.

\subsection{}
Unless otherwise noted, we assume throughout this section that $p=2$.  Further, we assume that (i) $G=SL_6$ (type $A_5$) and $\lambda = \varpi_1+\varpi_2+\varpi_4+\varpi_5$ or (ii) $G$ is of type $E_7$ and $\lambda = \varpi_4$.  In either case, $2\alpha_0$ appears as a weight in $L(\lambda)$.  Therefore $\St_1 \otimes L(\alpha_0)^{(1)}$ is a composition factor of $\St_1 \otimes L(\lambda)$.  At the same time, $\Delta(\alpha_0) \cong \mathfrak{g}$ (the adjoint representation), and $L(\alpha_0) \cong \mathfrak{g}/\mathfrak{z}(\mathfrak{g})$, with $\mathfrak{z}(\mathfrak{g})$ denoting the one-dimensional center of $\mathfrak{g}$.  Since $L(\alpha_0)^{(1)} \subsetneq \nabla(\alpha_0)^{(1)}$, by \cite[Proposition II.3.19]{rags},
$$\St_1 \otimes L(\alpha_0)^{(1)} \subsetneq \St_1 \otimes \nabla(\alpha_0)^{(1)} \cong \nabla(\rho+2\alpha_0),
$$
and so the composition factor $\St_1 \otimes L(\alpha_0)^{(1)}$ does not have a good filtration.  We note that $G$ is simply-laced, so that $\alpha_0$ is the highest root.  For $SL_6$, we have $\alpha_0 = \varpi_1+\varpi_5$, while, for $E_7$, $\alpha_0=\varpi_1$.

For $G=SL_6$, the other dominant weights $\gamma$ such that $2\gamma$ is a weight of $\Delta(\lambda)$ are
$$\gamma = 0, \varpi_3.$$
For $G$ of type $E_7$, they are
$$\gamma = 0, \varpi_7.$$
For $G=SL_6$, we have $L(\varpi_3)\cong \nabla(\varpi_3)$, and for $G$ of type $E_7$, we have $L(\varpi_7) \cong \nabla(\varpi_7)$.  Neither module (for the given $G$) extends nor can be extended by $L(\alpha_0)$.  Furthermore, we have in both cases that
$$\text{Ext}_G^1(k,L(\alpha_0)) \cong \text{Ext}_G^1(L(\alpha_0),k) \cong k.$$
These one-dimensional extension groups are accounted for by the indecomposable modules $\Delta(\alpha_0)$ and $\nabla(\alpha_0)$.  One may further check by standard long exact sequence computations that
$$\text{Ext}_G^1(\nabla(\alpha_0),L(\alpha_0)) = 0; \quad \text{Ext}_G^1(L(\alpha_0),\nabla(\alpha_0))=0;$$
$$\text{Ext}_G^1(k,\nabla(\alpha_0)) = 0; \quad \text{Ext}_G^1(\nabla(\alpha_0),k)  \cong k.$$
Applying the $\tau$-functor, which interchanges Weyl and induced modules, while preserving simple (and tilting) modules, we obtain all extensions involving $\Delta(\alpha_0)$, $k$, and $L(\alpha_0)$.  Summarizing, the collection of indecomposable $G$-modules having composition factors coming from the collection $\{k,L(\alpha_0)\}$ are (up to isomorphism)
$$\{k,\;L(\alpha_0),\;\nabla(\alpha_0),\; \Delta(\alpha_0),\;T(\alpha_0)\}.$$
The structure of the tilting module $T(\alpha_0)$ is given by the exact sequence
$$0 \rightarrow k \rightarrow T(\alpha_0) \rightarrow \nabla(\alpha_0) \rightarrow 0.$$
Via the equivalence of categories between $G$-mod and its Steinberg block, it follows that an indecomposable summand of $\St_1 \otimes L(\lambda)$ that contains $\St_1 \otimes L(\alpha_0)^{(1)}$ as a composition factor must be isomorphic to one of the following:
$$\{\St_1 \otimes L(\alpha_0)^{(1)},\St_1 \otimes \nabla(\alpha_0)^{(1)},\St_1 \otimes \Delta(\alpha_0)^{(1)},\St_1 \otimes T(\alpha_0)^{(1)}\}.$$
Note that, if we instead work with $\St_1 \otimes \Delta(\lambda)$, the only possibilities from this list are the two involving the Weyl module or the tilting module. One can 
also make the deduction about the module structures above by working with the truncated category obtained by looking at the full subcategory of rational $G$-modules having 
composition factors with highest weight less than or equal to $\alpha_{0}$ (and linked to $\alpha_{0}$). This category has finite representation type. 

\begin{lemma}\label{L:nosubmodule}
Assume $p = 2$. Let $G$ be of type $A_5$ with $\la = \varpi_1 + \varpi_2 + \varpi_4 + \varpi_5$ or $G$ be of type $E_7$ with $\la = \varpi_4$.  The summands of $\St_1 \otimes L(\lambda)$ containing the composition factor $\St_1 \otimes L(\alpha_0)^{(1)}$ all have a good filtration if and only if
$$\textup{Hom}_G(\St_1 \otimes L(\alpha_0)^{(1)},\St_1 \otimes L(\lambda)) = 0.$$
\end{lemma}

\begin{proof}
By the previous discussion, if $\textup{Hom}_G(\St_1 \otimes L(\alpha_0)^{(1)},\St_1 \otimes L(\lambda)) = 0$, then the only relevant summands that can appear in $\St_1 \otimes L(\lambda)$ are of the form $\St_1 \otimes T(\alpha_0)^{(1)}$ and $\St_1 \otimes \Delta(\alpha_0)^{(1)}$, thus these summands have a Weyl filtration.  Now, $\St_1 \otimes L(\lambda)$ is $\tau$-invariant.  Further, if $M$ is a summand of $\St_1 \otimes L(\lambda)$, then $^{\tau}M$ is isomorphic to a summand of $\St_1 \otimes L(\lambda)$ having the same composition factors.  It follows then that a summand containing the factor $\St_1 \otimes L(\alpha_0)^{(1)}$ also has a good filtration.

The converse is established by the reverse implication of each step in the argument.
\end{proof}

\subsection{}
In order to analyze this situation further, we need to establish some basic facts about tensor products of modules.  Note that the results in this subsection hold under our general assumption on $G$, and Lemmas \ref{L:FrobTensor} and \ref{L:simpletensorsum} hold for arbitrary primes $p$.  Recall that for a $B$-module $M$, there is an ``evaluation map'' $\varepsilon_M: \textup{ind}_B^G M \rightarrow M$ which induces one direction of the Frobenius reciprocity bijection \cite[I.3.4]{rags}.

\begin{lemma}\label{L:FrobTensor}
Let $M$ be a $B$-module and $N,N^{\prime}$ be $G$-modules.  Composition with the $B$-module homomorphism $\varepsilon_M \otimes \textup{id}$ defines a bijection
$$\textup{Hom}_G(N^{\prime},(\textup{ind}_B^G M) \otimes N) \xrightarrow{\sim} \textup{Hom}_B(N^{\prime},M \otimes N).$$
\end{lemma}

\begin{proof}
As recalled above, composition with $\varepsilon_{M \otimes N}$ defines a bijection
$$\textup{Hom}_G(N^{\prime},\textup{ind}_B^G (M \otimes N)) \xrightarrow{\sim} \textup{Hom}_B(N^{\prime},M \otimes N).$$
The ``tensor identity'' in \cite[I.3.6]{rags} is established by a canonical isomorphism
$$\textup{ind}_B^G (M \otimes N) \xrightarrow{\sim} (\textup{ind}_B^G M) \otimes N.$$
This isomorphism is specified via canonical embeddings
$$\textup{ind}_B^G (M \otimes N) \hookrightarrow M \otimes N \otimes k[G] \hookleftarrow (\textup{ind}_B^G M) \otimes N,$$
together with an automorphism of $M \otimes N \otimes k[G]$ that sends the embedding on the left isomorphically onto the embedding on the right.  Now, the morphisms $\varepsilon_M \otimes \text{id}$ and $\varepsilon_{M \otimes N}$ both come from these embeddings, by restricting the map
$$\text{id} \otimes \text{id} \otimes \varepsilon_G: M \otimes N \otimes k[G] \rightarrow M \otimes N \otimes k,$$
where $\varepsilon_G$ is the counit map on $k[G]$, to each embedded subgroup.  This proves the claim.
\end{proof}

\begin{lemma}\label{L:simpletensorsum}
Let $\mu,\lambda \in X(T)_+$, and let $v_{\mu}$ and $z_{\lambda}$ denote highest weight vectors of the modules $\Delta(\mu)$ and $L(\lambda)$ respectively.  Let $M$ be any $G$-module.  If
$$\phi: \Delta(\mu) \rightarrow M \otimes L(\lambda)$$
is a non-zero homomorphism of $G$-modules, then there is some $0 \ne m \in M$ such that
$$\phi(v_{\mu})=(m \otimes z_{\lambda}) + y, \quad \textup{with } y \in M \otimes \left(\sum_{\sigma< \lambda} L(\lambda)_{\sigma}\right).$$
\end{lemma}

\begin{proof}
There is a canonical inclusion
$$\text{Hom}_G(\Delta(\mu),M \otimes L(\lambda)) \hookrightarrow \text{Hom}_G(\Delta(\mu),M \otimes \nabla(\lambda)).$$
By Lemma \ref{L:FrobTensor}, the $B$-module homomorphism
$$M \otimes \nabla(\lambda) \xrightarrow{\text{id} \otimes \varepsilon_{\lambda}} M \otimes \lambda$$
induces a bijection
$$\text{Hom}_G(\Delta(\mu),M \otimes \nabla(\lambda)) \cong \text{Hom}_B(\Delta(\mu),M \otimes \lambda).$$
As a $B$-module, $\Delta(\mu)$ is generated by $v_{\mu}$, thus a non-zero $B$-homomorphism from $\Delta(\mu)$ to any $B$-module must send $v_{\mu}$ to a non-zero element.  We have a vector space decomposition
$$M \otimes \nabla(\lambda) = M \otimes z_{\lambda} +  M \otimes \left(\sum_{\sigma<\lambda} \nabla(\lambda)_{\sigma}\right).$$
The result then follows by noting that $\text{id}\otimes\varepsilon_{\la}$ sends
$$M \otimes \left(\sum_{\sigma<\lambda} \nabla(\lambda)_{\sigma}\right) \rightarrow 0.$$
\end{proof}

\begin{lemma}\label{L:maxvectordown}
Assume $p = 2$. Let $w_{\rho + 2\alpha_0}$ be a highest weight vector of $\St_1 \otimes \Delta(\alpha_0)^{(1)}$ and $w_{\rho}$ be a maximal vector generating the simple submodule $\St_1 \otimes k \le \St_1 \otimes \Delta(\alpha_0)^{(1)}$.  There is some $X \in \Dist(U)$ of weight $-2\alpha_0$ such that
$$X.w_{\rho + 2\alpha_0} = w_{\rho}.$$
The comultiplication of $X$ in $\Dist(U)$ is given by
$$\Delta(X)=X \otimes 1 + 1 \otimes X + \sum X_i^{\prime} \otimes X_i^{\prime\prime},$$
where for each $i$
$$-2\alpha_0<\textup{wt}(X_i^{\prime}),\textup{wt}(X_i^{\prime\prime})<0.$$
\end{lemma}

\begin{proof}
The Weyl module $\Delta(\rho+2\alpha_0) \cong \St_1 \otimes \Delta(\alpha_0)^{(1)}$ is generated over $B$, and over $U$, by any highest weight vector.  The same is true over $\Dist(B)$ and $\Dist(U)$, thus there is some $X \in \Dist(U)$ that gives the required action.  Moreover, it is clear that we can chose $X$ to be a $T$-weight vector of weight $-2\alpha_0$ (indeed, any $X$ such that $X.w_{\rho + 2\alpha_0} = w_{\rho}$ will be a sum of weight vectors, and any elements in the sum not having weight $-2\alpha_0$ must then act as zero, so we can modify $X$ by subtracting off if necessary such terms).

The augmentation ideal of $\Dist(U)$ is the vector subspace spanned by all $T$-weight vectors of weight $\ne 0$, hence $X$ is in this ideal.  The claim about $\Delta(X)$ then follows from a general fact about the comultiplication of elements in the augmentation ideal of a Hopf algebra, together with the fact that the terms in $\Delta(X)$ must have total weight $-2\alpha_0$.
\end{proof}

\subsection{}\label{S:Areductions}
Returning to our special assumptions on $G$, $p$, and $\la$, we now give a series of reductions toward proving that any summand of $\St_1 \otimes L(\lambda)$ containing $\St_1 \otimes L(\alpha_0)^{(1)}$ as a composition factor is tilting.

\bigskip
\noindent \textbf{Reduction 1:} By Lemma \ref{L:nosubmodule}, this is equivalent to showing that $\St_1 \otimes L(\alpha_0)^{(1)}$ does not appear as a submodule of $\St_1 \otimes L(\lambda)$.\\

\noindent \textbf{Reduction 2:} This is equivalent to showing that any maximal vector in $\St_1 \otimes L(\lambda)$ of weight $\rho+2\alpha_0$ generates a submodule isomorphic to $\St_1 \otimes \Delta(\alpha_0)^{(1)}$.\\

\noindent \textbf{Reduction 3:} By Lemma \ref{L:maxvectordown}, this is equivalent to showing that if $v_{\rho+2\alpha_0}$ is any such maximal vector of $\St_1 \otimes L(\lambda)$ and $X \in \Dist(U)$ is chosen as in the lemma, then $X.v_{\rho+2\alpha_0} \ne 0$.\\

\subsection{}
For this subsection, we restrict ourselves to the case $G=SL_6$ and $\lambda = \varpi_1+\varpi_2+\varpi_4+\varpi_5$. Fix, for each positive root $\beta$, the usual negative and positive Chevalley basis elements $f_{\beta}$ and $e_{\beta}$ of ${\mathfrak g}$ coming from the natural representation, and let $h_{\beta}=[e_{\beta},f_{\beta}]$.  We view these as elements inside $\Dist(G)$.

We have $[e_{\alpha_i},f_{\alpha_j}] = 0$ if $i \ne j$.  Because $p=2$, we also have 
$$[h_{\alpha_i},e_{\alpha_j}] = e_{\alpha_j} \qquad \text{and} \qquad [h_{\alpha_i},f_{\alpha_j}] = f_{\alpha_j} \qquad \text{if $\abs{i-j}=1$},$$
otherwise
$$[h_{\alpha_i},e_{\alpha_j}] = 0 = [h_{\alpha_i},f_{\alpha_j}].$$

One computes that
$$\lambda-2\alpha_0=\alpha_2+\alpha_3+\alpha_4.$$
Let $J=\{\alpha_2,\alpha_3,\alpha_4\}$. By \cite[II.2.11]{rags} (see \cite[Proposition 3.3]{GGN} for a more thorough discussion), the $L_J$-Weyl module $\Delta_J(\lambda)$ is a direct summand of $\Delta(\lambda)$, considered as an $L_J$-module.  More specifically
$$\Delta_J(\lambda) \cong \sum_{\lambda-\mu \in \mathbb{Z}J} \Delta(\lambda)_{\mu},$$
with the $L_J$-complement to $\Delta_J(\lambda)$ consisting of the sum of the remaining weight spaces.  If we further restrict our attention to the weight spaces from $2\alpha_0$ to $\lambda$, it follows (by weight considerations) that there is an isomorphism of $B^+$-modules
$$\sum_{\mu \ge 2\alpha_0} \Delta(\lambda)_{\mu} \cong \sum_{\mu \ge 2\alpha_0} \Delta_J(\lambda)_{\mu}.$$

The derived subgroup $(L_J,L_J)$ is isomorphic to $SL_4$, the restriction of $\Delta_J(\lambda)$ to $(L_J,L_J)$ is the adjoint representation, and $2\alpha_0$ restricts to the zero weight for $T \cap (L_J,L_J)$.  Using the structure of $\Lie(SL_4)$, one readily computes the $B^+$ structure of $\sum_{\mu \ge 2\alpha_0} \Delta(\lambda)_{\mu}$.

In particular, given a maximal weight vector $z_{\la}$ of $\Delta(\la)$, the following is a $T$-basis for $\Delta(\lambda)_{2\alpha_0}$:
$$\{f_{\alpha_2}f_{\alpha_3}f_{\alpha_4}.z_{\lambda}, \quad f_{\alpha_4}f_{\alpha_3}f_{\alpha_2}.z_{\lambda}, \quad f_{\alpha_3}f_{\alpha_4}f_{\alpha_2}.z_{\lambda} \}.$$
We also have $(f_{\alpha_2}f_{\alpha_3}f_{\alpha_4}.z_{\lambda}+f_{\alpha_4}f_{\alpha_3}f_{\alpha_2}.z_{\lambda})$ as a $U^+$-fixed vector, and
\begin{equation}\label{E:L_Jquotient}
L(\lambda)_{2\alpha_0} \cong \Delta(\lambda)_{2\alpha_0}/(f_{\alpha_2}f_{\alpha_3}f_{\alpha_4}.z_{\lambda}+f_{\alpha_4}f_{\alpha_3}f_{\alpha_2}.z_{\lambda}).\end{equation}

Let $w_{\rho}$ be a highest weight vector in $\St_1$.  Any $T$-basis $\{y_i\}$ for $\Dist(U_1)$ yields a $T$-basis $\{y_i.w_{\rho}\}$ for $\St_1$.  While a standard choice is to take a PBW-basis after choosing an ordering of roots, one can alternatively take a basis consisting of products of the various $f_{\alpha_i}$ for $\alpha_i$ a simple root.  While it is in general harder to list all such basis elements in this manner, we will find it more convenient in our limited consideration.

\begin{lemma}\label{L:actionofe}
Assume $G$ is $SL_6$ and $p = 2.$  Let $w_{\rho}$ be a a highest weight vector in $\St_1$.  Let $\alpha_{i_1},\ldots,\alpha_{i_m}$ be an ordered collection of simple roots without repetition (so $1 \leq m \leq 5$).  For any simple root $\alpha$, 
$$e_{\alpha}f_{\alpha_{i_1}}\cdots f_{\alpha_{i_m}}.w_{\rho}=0$$
if $\alpha_{i_j}\ne \alpha$ for all $i_j$.  Otherwise, if some $\alpha_{i_j}=\alpha$, then
$$e_{\alpha}f_{\alpha_{i_1}}\cdots f_{\alpha_{i_m}}.w_{\rho}=(s+1)f_{\alpha_{i_1}}\cdots f_{\alpha_{i_{j-1}}} f_{\alpha_{i_{j+1}}} \cdots f_{\alpha_{i_m}}.w_{\rho},$$
where $s$ is the number of $i_k$ such that $k>j$ and $\abs{i_j-i_k}=1$.
\end{lemma}

\begin{proof}
If no $\alpha_{i_j}=\alpha$, then $e_{\alpha}$ commutes past each $f_{\alpha_{i_j}}$ in $\Dist(G_1)$, and since $e_{\alpha}$ annihilates $w_{\rho}$, the first statement follows.

Otherwise, if $\alpha=\alpha_{i_j}$ for some $j$, then $e_{\alpha}$ commutes past each $f_{\alpha_{i_{\ell}}}$, $\ell < j$.  One then applies the commutation relations above.  We have
\begin{align*}
e_{\alpha}f_{\alpha_{i_1}}\cdots f_{\alpha_{i_m}}.w_{\rho} 
&= f_{\alpha_{i_1}}\cdots f_{\alpha_{i_{j-1}}}e_{\alpha}f_{\alpha_{i_j}}f_{\alpha_{i_{j+1}}}\cdots f_{\alpha_{i_m}}.w_{\rho}\\
&= f_{\alpha_{i_1}}\cdots f_{\alpha_{i_{j-1}}}(f_{\alpha_{i_j}}e_{\alpha}+h_{\alpha_{i_j}})f_{\alpha_{i_{j+1}}}\cdots f_{\alpha_{i_m}}.w_{\rho}\\
&= f_{\alpha_{i_1}}\cdots f_{\alpha_{i_{j-1}}}f_{\alpha_{i_j}}e_{\alpha}f_{\alpha_{i_{j+1}}}\cdots f_{\alpha_{i_m}}.w_{\rho} + f_{\alpha_{i_1}}\cdots f_{\alpha_{i_{j-1}}}h_{\alpha_{i_j}}f_{\alpha_{i_{j+1}}}\cdots f_{\alpha_{i_m}}.w_{\rho}\\
&= f_{\alpha_{i_1}}\cdots f_{\alpha_{i_{j-1}}}h_{\alpha_{i_j}}f_{\alpha_{i_{j+1}}}\cdots f_{\alpha_{i_m}}.w_{\rho},
\end{align*}
since the first term is seen to be zero, by commuting the $e_{\alpha}$ past the remaining terms.  We now use the fact that
$$h_{\alpha_{i_j}}f_{\alpha_{i_k}}=f_{\alpha_{i_k}}h_{\alpha_{i_j}}+f_{\alpha_{i_k}}=f_{\alpha_{i_k}}(h_{\alpha_{i_j}}+1)$$
if $\abs{i_j-i_k}=1$, otherwise
$$h_{\alpha_{i_j}}f_{\alpha_{i_k}}=f_{\alpha_{i_k}}h_{\alpha_{i_j}}.$$
Repeatedly applying this we obtain 
$$
e_{\alpha}f_{\alpha_{i_1}}\cdots f_{\alpha_{i_m}}.w_{\rho} = 
f_{\alpha_{i_1}}\cdots f_{\alpha_{i_{j-1}}}f_{\alpha_{i_{j+1}}}\cdots f_{\alpha_{i_m}}(h_{\alpha_{i_j}}+s).w_{\rho},$$
and as $h_{\alpha_{i_j}}.w_{\rho}=w_{\rho}$, the result follows.
\end{proof}

\begin{lemma}\label{L:basis}
Assume $G$ is $SL_6$ and $p = 2$. The following vectors form a basis of maximal vectors in $\St_1 \otimes \Delta(\lambda)$ having weight $\rho+2\alpha_0$:
$${\bf{v}_{1}}=w_{\rho} \otimes f_{\alpha_3}f_{\alpha_2}f_{\alpha_4}.z_{\lambda} + f_{\alpha_2}.w_{\rho} \otimes f_{\alpha_3}f_{\alpha_4}.z_{\lambda} + f_{\alpha_4}.w_{\rho} \otimes f_{\alpha_3}f_{\alpha_2}.z_{\lambda}$$
$$ + f_{\alpha_2}f_{\alpha_3}.w_{\rho} \otimes f_{\alpha_4}.z_{\lambda} + f_{\alpha_4}f_{\alpha_3}.w_{\rho} \otimes f_{\alpha_2}.z_{\lambda}$$
$$ + (f_{\alpha_2}f_{\alpha_3}f_{\alpha_4} + f_{\alpha_4}f_{\alpha_3}f_{\alpha_2}).w_{\rho} \otimes z_{\lambda},$$

$${\bf{v}_{2}}=w_{\rho} \otimes f_{\alpha_2}f_{\alpha_3}f_{\alpha_4}.z_{\lambda} + f_{\alpha_3}.w_{\rho} \otimes f_{\alpha_2}f_{\alpha_4}.z_{\lambda} + f_{\alpha_3}f_{\alpha_2}.w_{\rho} \otimes f_{\alpha_4}.z_{\lambda}$$
$$+ f_{\alpha_3}f_{\alpha_4}.w_{\rho} \otimes f_{\alpha_2}.z_{\lambda} + f_{\alpha_3}f_{\alpha_2}f_{\alpha_4}.w_{\rho} \otimes z_{\lambda},$$

$${\bf{v}_{3}}=w_{\rho} \otimes (f_{\alpha_2}f_{\alpha_3}f_{\alpha_4}+f_{\alpha_4}f_{\alpha_3}f_{\alpha_2}).z_{\lambda}.$$
\end{lemma}

\begin{proof}
The elements
$$f_{\alpha_3}f_{\alpha_2}f_{\alpha_4}.z_{\lambda}, \quad f_{\alpha_2}f_{\alpha_3}f_{\alpha_4}.z_{\lambda}, \quad \text{and} \quad (f_{\alpha_2}f_{\alpha_3}f_{\alpha_4}+f_{\alpha_4}f_{\alpha_3}f_{\alpha_2}).z_{\lambda}$$
are linearly independent in $\Delta(\lambda)$.  From this it follows that the three different sums of simple tensors listed above are linearly independent in $\St_1 \otimes \Delta(\lambda)$.

To verify maximality, since $(\rho+\lambda)-(\rho+2\alpha_0) = \alpha_2+\alpha_3+\alpha_4$, we need only check that each $e_{\alpha_i}$, $2 \le i \le 4$, annihilates these elements, where the action on a simple tensor is via $e_{\alpha_i} \otimes 1 + 1 \otimes e_{\alpha_i}$.  The verification for ${\bf{v}_{3}}$ is immediate as it is annihilated both by $e_{\alpha_i} \otimes 1$ and by $1 \otimes e_{\alpha_i}$.  For ${\bf{v}_{1}},{\bf{v}_{2}}$, one applies Lemma \ref{L:actionofe} to see that the sum $e_{\alpha_i} \otimes 1 + 1 \otimes e_{\alpha_i}$ annihilates each vector.
\end{proof}

Fix a surjective $G$-module homomorphism $f:\Delta(\lambda) \rightarrow L(\lambda)$.  Over $L_J$, $f$ restricts to a surjective homomorphism $\Delta_J(\lambda) \rightarrow L_J(\lambda)$.  By (\ref{E:L_Jquotient}), it follows that
$$(\text{ker} \, f) \cap \Delta(\lambda)_{2\alpha_0} = \text{Span}\{f_{\alpha_2}f_{\alpha_3}f_{\alpha_4}.z_{\lambda}+f_{\alpha_4}f_{\alpha_3}f_{\alpha_2}.z_{\lambda}\}.$$
We obtain from $f$ a surjective $G$-module homomorphism
$$\text{id} \otimes f: \St_1 \otimes \Delta(\lambda) \rightarrow \St_1 \otimes L(\lambda).$$
On a simple tensor of the form $w_{\rho} \otimes v$ we have $(\text{id} \otimes f)(w_{\rho} \otimes v)=w_{\rho} \otimes f(v)$.  Given the definition of $f$ on $\Delta(\lambda)_{2\alpha_0}$, it follows that $(\text{id} \otimes f)({\bf{v}_{3}}) = 0$, and that
$(\text{id} \otimes f)({\bf{v}_{1}})$ and  $(\text{id} \otimes f)({\bf{v}_{2}})$ are linearly independent.

\begin{lemma}
Assume $G$ is $SL_6$ and $p = 2$. Let $X \in \Dist(U)$ be as in Lemma \ref{L:maxvectordown} and ${\bf{v}_i}$ be as in Lemma \ref{L:basis}.  Then, in $\St_1\otimes L(\la)$, 
$$X.(\textup{id} \otimes f)({\bf{v}_{1}}) \ne 0, \quad X.(\textup{id} \otimes f)({\bf{v}_{2}}) \ne 0.$$
Consequently, $(\St_1 \otimes \Delta(\alpha_0)^{(1)})^{\oplus 2} \hookrightarrow \St_1 \otimes L(\lambda)$.
\end{lemma}

\begin{proof}
The module $\St_1 \otimes \Delta(\lambda)$, being the tensor product of two Weyl modules, has a Weyl filtration.  Therefore, the vectors ${\bf{v}_{1}},{\bf{v}_{2}},$ and ${\bf{v}_{3}}$  each generate copies of $\St_1 \otimes \Delta(\alpha_0)^{(1)}$.  It follows then that $X.{\bf{v}_{i}} \ne 0$, for $1\le i \le 3$.  Further, since the ${\bf{v}_{i}}$ are linearly independent, the elements $X.{\bf{v}_{i}}$ must be also (the ${\bf{v}_{i}}$ must together generate a submodule isomorphic to $(\St_1 \otimes \Delta(\alpha_0)^{(1)})^{\oplus 3}$).  Thus each $X.{\bf{v}_{i}}$ is a maximal vector in $\St_1 \otimes \Delta(\lambda)$ of weight $\rho$ (lying within a copy of $\St_1$ inside the $G$-socle of $\St_1 \otimes \Delta(\la)$).  By Lemma \ref{L:simpletensorsum}, it follows that each $X.{\bf{v}_{i}}$, expressed as a sum of simple tensors in $\St_1 \otimes \Delta(\lambda)$, necessarily involves a term of the form $w_{\rho} \otimes v_0$, with $0 \ne v_0 \in \Delta(\lambda)_0$.  On the other hand, multiplying the expression for the comultiplication of $X$ in Lemma \ref{L:maxvectordown} against the explicit computations of the ${\bf{v}_{i}}$ above, for every $X.{\bf{v}_{i}}$ to contain a term of this form it must hold that 
$$w_{\rho} \otimes Xf_{\alpha_3}f_{\alpha_2}f_{\alpha_4}.z_{\lambda}  \ne 0,\quad w_{\rho} \otimes Xf_{\alpha_2}f_{\alpha_3}f_{\alpha_4}.z_{\lambda}  \ne  0, \quad w_{\rho} \otimes X(f_{\alpha_2}f_{\alpha_3}f_{\alpha_4}+f_{\alpha_4}f_{\alpha_3}f_{\alpha_2}).z_{\lambda}   \ne  0.$$
There is a non-zero $B$-module (composite) homomorphism
$$\St_1 \rightarrow \Delta(\lambda) \otimes \nabla(\rho-\lambda) \rightarrow \Delta(\lambda) \otimes (\rho-\lambda)$$
which restricts to a surjective $U$-module homomorphism $\St_1 \rightarrow \Delta(\lambda)$ (this holds more generally).  In particular, any element in $\Dist(U)$ that does not annihilate $z_{\lambda}$ cannot annihilate $w_{\rho}$.  It follows that
$$Xf_{\alpha_3}f_{\alpha_2}f_{\alpha_4}.w_{\rho} \otimes z_{\lambda}  \ne 0,\quad  X(f_{\alpha_2}f_{\alpha_3}f_{\alpha_4}+f_{\alpha_4}f_{\alpha_3}f_{\alpha_2}).w_{\rho} \otimes z_{\lambda}   \ne  0.$$
Since $f(z_{\lambda}) \ne 0$, then
$$(\textup{id} \otimes f)(X.{\bf{v}_{1}}) \ne 0, \quad (\textup{id} \otimes f)(X.{\bf{v}_{2}}) \ne 0,$$
as each element, as a sum of simple tensors in $\St_1 \otimes L(\lambda)$, contains a non-zero term that cannot be canceled out by the other terms. 
\end{proof}

By our reductions in Section \ref{S:Areductions}, it follows that $\St_1\otimes L(\la)$ has a good filtration. As discussed in Section \ref{S:A5}, this case was the only remaining one to check for $SL_6$, which proves the following.

\begin{theorem}\label{T:typeA5}
Assume $G = SL_6$ and $p = 2$.  Then
\begin{itemize}
\item[(a)] $\St_r\otimes L(\la)$ has a good filtration for all $\la \in X_r$, $r \geq 1$.
\item[(b)] In this case Conjecture~\ref{donkinconj}($\Rightarrow$) holds.
\end{itemize}
\end{theorem}

\section{Tensoring With $k[G_r]$} 

\subsection{} Recall that given any affine scheme $X$ over $k$ and an algebraic action of $G$ on $X$, the coordinate ring $k[X]$ becomes a $G$-module by the action $g.f(x)=f(g^{-1}.x)$ (such actions are intended to hold functorially, i.e., for every $k$-algebra $A$, every $g \in G(A), x \in X(A),$ and $f \in A[X_A]$).  In particular, $k[G]$ becomes a $G$-module via the conjugation action of $G$ on itself.  Further, since $G_r$ is a normal subgroup scheme of $G$, one similarly obtains a $G$-module structure on $k[G_r]$.

There is a $(G \times G)$-action on $G$ via $(g_1,g_2).h=g_1hg_2^{-1}$.  If we take the diagonal embedding of $G$ into $(G \times G)$, then the restriction of this action is the conjugation action of $G$ on itself.  We recall the following result due to Donkin and Koppinen \cite[Proposition II.4.20]{rags} (see also \cite[Remark II.4.21]{rags}).

\begin{proposition}
As a $(G \times G)$-module, $k[G]$ has a good filtration.  The good filtration factors are of the form
$$\nabla(\lambda)\otimes \nabla(\lambda^{*}), \qquad \lambda \in X_+,$$
each occuring with multiplicity one.  In particular, $k[G]$, as a $G$-module under the action induced by the conjugation action of $G$, has a good filtration.
\end{proposition}

\subsection{} In consideration of this result and Conjecture \ref{donkinconj}, we ask the following question 
that has been answered in the affirmative when $p$ is good by Donkin \cite{donkin18}. 
\begin{quest}
For each $r \ge 1$, does $\St_r \otimes k[G_r]$ have a good filtration?
\end{quest}

We conclude by making the following observation.  Let $I_{\varepsilon}$ denote the augmentation ideal of $k[G]$.  Under the conjugation action of $G$, this ideal is a $G$-submodule of $k[G]$, with $G$-module complement spanned by the subalgebra $k\cdot 1 \le k[G]$.  It follows then that $I_{\varepsilon}$ has a good filtration over $G$.  For each $r \ge 1$, let $M_r$ denote the subset of $k[G]$ defined as
$$M_r = \{ x^{p^r} \mid x \in I_{\varepsilon}\}.$$
Since $k$ is algebraically closed and of characteristic $p$, this is a subspace of $k[G]$, and indeed a submodule for $G$.  Further, the ideal defining the closed subgroup scheme $G_r$ is the ideal generated by $M_r$.  Thus, the algebra homomorphism $k[G] \rightarrow k[G_r]$ factors (as $G$-modules) through $k[G]/M_r$.

\begin{theorem}
The $G$-module $\St_r \otimes (k[G]/M_r)$ has a good filtration.
\end{theorem}

\begin{proof}
If $X$ and $Y$ are any two affine schemes with algebraic $G$-actions, and if $f:X \rightarrow Y$ is a $G$-equivariant homomorphism of schemes, then the comorphism $f^*:k[Y] \rightarrow k[X]$ is a homomorphism of $G$-modules (with the induced $G$-action on coordinate rings as above).

Now let $F:G \rightarrow G$ be the Frobenius morphism arising from a chosen split $\mathbb{F}_p$-structure on $G$, and let $F^r$ denote its $r$-th iterate.  Consider $(F^r)^*: k[G] \to k[G]$.  We claim that
$$(F^r)^*(I_{\varepsilon}) = M_r.$$
To see this, let $\mathbb{F}_p[G]$ denote the chosen $\mathbb{F}_p$-structure of $k[G]$ and $I_{\varepsilon,\mathbb{F}_p}$ denote the augmentation ideal of this Hopf algebra.  The map $(F^r)^*$ then arises (using this chosen $\mathbb{F}_p$-structure) as simply the $p^r$-th power map on $\mathbb{F}_p[G]$.  If we fix an $\mathbb{F}_p$-basis $\{x_i\}$ for $I_{\varepsilon,\mathbb{F}_p}$, then $\{x_i \otimes_{\mathbb{F}_p} 1\}$ is a $k$-basis for $I_{\varepsilon}$.  We have that
$$(F^r)^*(x_i \otimes_{\mathbb{F}_p} 1) = {x_i}^{p^r} \otimes_{\mathbb{F}_p} 1 = ({x_i} \otimes_{\mathbb{F}_p} 1)^{p^r}.$$
From this the claim follows. Further, since $G$ is reduced, $(F^r)^*$ is injective, and so $(F^r)^*$ maps $I_{\varepsilon}$ bijectively to $M_r$.

Now, the homomorphism $F^r:G \rightarrow G$ is $G$-equivariant, if $G$ acts on the target via $F^r$.  Therefore, the comorphism $(F^r)^*$ is a $G$-module homomorphism.  It follows then that, as a $G$-module, $M_r \cong {I_{\varepsilon}}^{(r)}$.
Applying \cite[II.3.19]{rags}, it follows that $\St_r \otimes M_r$ has a good filtration.  Since $\St_r \otimes k[G]$ also has a good filtration, it then follows from \cite[II.4.17]{rags} (which is a consequence of Theorem \ref{T:cohcrit}) that $\St_r \otimes (k[G]/M_r)$ has a good filtration.
\end{proof}

For each $\lambda \in X_+$, $\nabla(\lambda)^{(r)} \subseteq \nabla(p^r\lambda)$.  In particular,
$$\nabla(\lambda)^{(r)} \otimes \nabla(\lambda^{*})^{(r)} \subseteq \nabla(p^r\lambda) \otimes \nabla(p^r\lambda^{*}).$$
Note that the proof of the preceding result shows that there is a submodule of $k[G]$, namely $M_r$, having a filtration over $G$ with sections of the form $\nabla(\lambda)^{(r)} \otimes \nabla(\lambda^{*})^{(r)}$, $0 \ne \lambda \in X_+$.


\providecommand{\bysame}{\leavevmode\hbox
to3em{\hrulefill}\thinspace}

\end{document}